\documentclass[11pt, reqno]{amsart}
\usepackage{graphicx, amssymb, amsmath, amsthm}
\usepackage{epsfig}
\usepackage{hyperref}
\usepackage{mathrsfs}
\numberwithin{equation}{section}
\usepackage{amsthm}
\usepackage{subfigure}
\usepackage{tikz}
\usepackage{here}
\usepackage{float}
\usepackage{pifont}
\usepackage{enumitem}

\usetikzlibrary{matrix,arrows}
\usetikzlibrary{shapes}
\usetikzlibrary{calc}
\usetikzlibrary{arrows}
\usetikzlibrary{decorations.pathreplacing,decorations.markings}

\usetikzlibrary{patterns}

\def\p{\partial}

\newtheorem{theorem}{Theorem}[section]

\newtheorem{lemma}[theorem]{Lemma}

\newtheorem{corollary}[theorem]{Corollary}

\newtheorem{proposition}[theorem]{Proposition}

\theoremstyle{definition}
\newtheorem{definition}[theorem]{Definition}
\newtheorem{remark}[theorem]{Remark}

\newtheorem*{rem}{Remark}
\makeatletter
\newcommand{\Extend}[5]{\ext@arrow0099{\arrowfill@#1#2#3}{#4}{#5}}
\makeatother

\begin{document}
\title[positive scalar Curvature II]{contractible $3$-manifolds and positive scalar curvature (II)}

\author[Jian. Wang]{Jian Wang}

\address{Universit\"at Augsburg, Institut f\"ur Mathematik, Lehrstuhl f\"ur Differentialgeometrie, Universit\"atsstr. 14, 86159 Augsburg, Germany}

\address{Universit\'e Grenoble Alpes, Institut Fourier, 100 rue des maths, 38610 Gi\`eres, France}
\email{jian.wang.4@stonybrook.edu}
\maketitle

\begin{abstract} In this article, we are interested in the question whether any complete contractible $3$-manifold of positive scalar curvature is homeomorphic to $\mathbb{R}^{3}$. We study the fundamental group at infinity, $\pi_{1}^{\infty}$, and its relationship with the existence of complete metrics of positive scalar curvature. We prove that a complete contractible $3$-manifold with positive scalar curvature and trivial $\pi^{\infty}_{1}$ is homeomorphic to $\mathbb{R}^{3}$. 
\end{abstract}

\section{Introduction}
This paper is the sequel of \cite{W1} and also devoted to the study of contractible $3$-manifolds which carry complete metrics of positive scalar curvature. We  are mainly concerned with the following question: 

\vspace{2mm}

\noindent\textbf{Question:} \emph{Is any complete contractible $3$-manifold of positive scalar curvature homeomorphic to $\mathbb{R}^{3}$} ?

\vspace{2mm}

Gromov-Lawson \cite{GL} and Chang-Weinberger-Yu\cite{CWY} independently proved that a complete contractible $3$-manifold with uniformly positive scalar curvature (i.e. the scalar curvature is away from zero) is homeomorphic to $\mathbb{R}^3$. The proof of Gromov and Lawson uses minimal surfaces theory while Chang, Weinberger and Yu use K-theory.

The topological structure of contractible $3$-manifolds is quite complicated. For example,  Whitehead \cite{Wh} and McMillan \cite{Mc} showed that  there are infinitely many mutually non-diffeomorphic contractible $3$-manifolds, such as the Whitehead manifold.

The Geometrisation conjecture which was confirmed by G. Perelman \cite{P1, P2, P3} and a result of McMillan  \cite{Mc1} tell us  that a contractible $3$-manifold can be written as an ascending union of handlebodies. Remark that if there are infinitely many handlebodies of genus zero (i.e. 3-balls), the $3$-manifold is homeomorphic to $\mathbb{R}^{3}$.

In \cite{W1},  we considered a contractible genus one $3$-manifold, an ascending union of solid tori. As mentioned above, $\mathbb{R}^{3}$ is not genus one but genus zero, since it is an increasing union of $3$-balls. In \cite{W1}, it was proved that no contractible genus one $3$-manifold admits a complete metric of nonnegative scalar curvature. 

In the present paper, we study the existence of complete metrics of positive scalar curvature and its relationship with the fundamental group at infinity.

The \emph{fundamental group at infinity}, $\pi_{1}^{\infty}$, of a path-connected space is the inverse limit of the fundamental groups of complements of compact subsets. (See Definition \ref{fgi}) 
The triviality of the fundamental group at infinity is not equivalent to the simply-connectedness at infinity. For example, the Whitehead manifold is not simply-connected at infinity but its fundamental group at infinity is trivial. 

We prove the following theorem:   \begin{theorem}\label{A} A complete contractible  3-manifold with positive scalar curvature and trivial $\pi_{1}^{\infty}$ is homeomorphic to $\mathbb{R}^{3}$.
\end{theorem}

%\begin{corollary} A complete contractible $3$-manifold with nonnegative scalar curvature and trivial $\pi^\infty_{1}$ is homeomorphic to $\mathbb{R}^3$. 
%\end{corollary}

However, there are uncountably many mutually non-homeomorphic contractible $3$-manifolds with non-trivial $\pi_{1}^\infty$. In Appendix C, we construct such a manifold and show that this manifold has no complete metric of positive scalar curvature.

%%We also study an open $3$-manifold of bounded genus (at most g), that is, an increasing union of  handlebodies whose genus are not greater than g (See Definition \ref{bounded}). In this case, its fundamental group at infinity may be non-trivial (See Appendix II). Note that there  also exists a contractible $3$-manifold whose genus is  unbounded but fundamental group is  trivial (See Appendix II). 

 %%Using a new topological surgery, called a \emph{Plane Surgery}, we establish  the following: 
%%\begin{theorem}\label{B} A contractible 3-manifold of positive scalar curvature and bounded genus is homeomorphic to $\mathbb{R}^{3}$.
%%\end{theorem}

\vspace{2mm}

\emph{1.1 Handlebodies and Property (H)} Let $(M, g)$ be a complete contractible $3$-manifold   of positive scalar curvature. It is an increasing union of closed handlebodies $\{N_{k}\}$ (see Theorem \ref{handle}).

 In the following, we consider that $M$ is not homeomorphic to $\mathbb{R}^{3}$.  
 We may assume that none of the  $N_{k}$ is contained in a $3$-ball (i.e. homeomorphic to a unit ball in $\mathbb{R}^3$) (see Remark \ref{ball}).
  It plays a crucial role in our argument. 

In the genus one case,  the family $\{N_{k}\}$ has several good properties. 
For example, the maps $\pi_{1}(\p N_{k})\rightarrow \pi_{1}(\overline{M\setminus N_{k}})$ and $\pi_{1}(\p N_{k})\rightarrow \pi_{1}(\overline{N_{k}\setminus N_{0}})$ are both injective (see Lemma 2.10 in \cite{W1}). 
These properties are crucial and necessary in the study of the existence of complete metrics of positive scalar curvature. Generally, the family $\{N_{k}\}$ may not have the above properties.

For example, the map $\pi_1(\p N_0)\rightarrow \pi_{1}(\overline{M\setminus N_{0}})$ may not be injective. To overcome it, we use topological surgeries on $N_{0}$ and find a new handlebody to replace it. Roughly, we use the loop lemma to find an embedded disc $(D, \p D)\subset (\overline{M\setminus N_{0}}, \p N_0)$ whose boundary is a non-nullhomotopic circle in $\p N_{0}$.  The new handlebody is obtained from $N_{0}$ by attaching a closed tubular neighborhood $N_{\epsilon}(D)$ of $D$ in $\overline{M\setminus N_{0}}$. 

 We repeatedly use topological surgeries on each $N_{k}$ to obtain a new family $\{R_{k}\}_{k}$ of closed handlebodies with  the following properties, called \emph{Property (H)}:

%each handlebody $N_{k}$ is of genus one (and has a unique handle). The topological information is fully determined by the behavior (for example, the geometric index) of these handles. In the general case, each $N_{k}$ may have several handles.  Some handles may have no contribution to the topological information and yield  technical difficulties for our argument.  Precisely , for some $k$, some handle of $N_{k}$ may be contained in a $3$-ball. 
%
%For example, if the map $\pi_{1}(N_{0})\rightarrow \pi_{1}(\overline{M\setminus N_{0}})$ is not injective, the loop lemma (See Lemma \ref{loop}) tells that there is an embedded disc $(\text{Int} D, \partial D)\subset ({M\setminus N_{0}}, \partial N_{0})$. The boundary $\gamma:=\partial D$ is a simple closed curve  which is not homotopically trivial in $\partial N_{0}$. The handle determined by the circle $\gamma$ is contained in a $3$-ball. 
%
%In order to kill this useless handle, we apply topological surgeries on $N_{0}$. Precisely, we construct a new handlebody to replace it. The new handlebody is obtained from $N_{0}$ by attaching a closed tubular neighborhood $N_{\epsilon}(D)$ of $D$ in $\overline{M\setminus N_{0}}$.  

%As above, we use these types of surgeries on each $N_{k}$ to obtain a new increasing family $\{R_{k}\}$ of closed handlebodies  with the following property, called \emph{Property (H)}: 

\begin{enumerate}
\item the map $\pi_{1}(\partial R_{k})\rightarrow \pi_{1}(\overline{R_{k}\setminus R_{0}})$ is injective for $k>0$;
\item the map $\pi_{1}(\partial R_{k})\rightarrow \pi_{1}(\overline{M\setminus R_{k}})$ is injective for $k\geq 0$;
\item each $R_{k}$ is homotopically trivial  in $R_{k+1}$ but not contained in a $3$-ball in $M$ ;
\item there exists a sequence $\{j_{k}\}_{k}$ of increasing integers , such that $\pi_{1}(\partial R_{k}\cap \partial N_{j_{k}})\rightarrow \pi_{1}(\partial R_{k})$ is surjective.
\end{enumerate}

\begin{rem} If $M$ is not homeomorphic to $\mathbb{R}^{3}$, the existence of such a family is ensured by Theorem \ref{H}. It is not unique. In addition, the union of such a family may not be equal to $M$.\end{rem}For example, if $M:=\cup_{k} N_{k}$ is a contractible genus one $3$-manifold, the family $\{N_{k}\}$ (assumed as above) satisfies \emph{property $(H)$}(see Lemma 2.10 in \cite{W1}).

\emph{1.2 The Vanishing Property} It is classical that the geometry of minimal surfaces gives topological informations on $3$-manifolds. 
This fact appeared in Schoen-Yau's works \cite{SY, SY1} as well as in Gromov-Lawson's \cite{GL}. 

In the genus one case, the geometry of a stable minimal surface is constrained by the geometric index (see \emph{Property $P$} in \cite{W1}). In  the higher genus case, the geometry of  stable minimal surfaces is related to  the fundamental group at infinity. 

In order to clarify their relationship, let us introduce a geometric property, called \emph{the Vanishing property}. Consider a complete contractible $3$-manifold $(M, g) $ of positive scalar curvature which is not homeomorphic to $\mathbb{R}^{3}$. As indicated above, there is an increasing family $\{R_{k}\}_{k}$ of closed handlebodies  with \emph{Property (H)}(see Theorem \ref{H}).

A complete embedded stable minimal surface $\Sigma\subset (M, g)$ is called to satisfy \emph{the Vanishing property} with respect to  the  family $\{R_{k}\}_{k}$ if there is a positive integer $k(\Sigma)$ so that for $k\geq k(\Sigma)$, any circle in $\Sigma\cap \partial R_{k}$ is null-homotopic in $\partial R_{k}$ (see Definition \ref{Vanishing}).

If a complete stable minimal  surface does not satisfy the Vanishing property with respect to  $\{R_{k}\}_{k}$, it gives a non-trivial element in the fundamental group at infinity  (see Lemma \ref{van1}). As a consequence, if $\pi_{1}^{\infty}$ is trivial, each complete stable minimal surface in $M$ has the Vanishing property with respect to  $\{R_{k}\}_{k}$ (see Corollary \ref{trivial-van1}).

\vspace{2mm}

\emph{1.3 The idea of the proof of Theorem \ref{A}}  Our main strategy is to argue by contradiction. Suppose that  a complete contractible 3-manifold $(M, g)$ with positive scalar curvature and trivial $\pi^{\infty}_{1}(M)$ is not homeomorphic to $\mathbb{R}^{3}$.

Before constructing minimal surfaces, let us introduce two notions from 3-dimensional topology. For a closed handlebody $N$ of genus $g>0$, a \emph{meridian} $\gamma\subset \partial N$ of $N$ is an embedded circle which is null-homotopic in $N$ but non-contractible in $\partial N$ (see Definition \ref{mer}).

 \emph{A system of meridians}  of  $N$ is a collection of $g$ distinct meridians $\{\gamma^{l}\}^{g}_{l=1}$ with the property that $\partial N \setminus \amalg_{l=1}^g\gamma^{l}$ is homeomorphic to an open disc with some punctures.
 Its existence is ensured by Lemma \ref{sys}. 
 
Let $\{N_{k}\}_{k}$ and $\{R_{k}\}_{k}$ be assumed as above. Since $N_{0}$ is not contained in a $3$-ball (see Remark \ref{ball}), the genus of $N_{k}$ is greater than zero. The handlebody $N_{k}$ has a system of meridians $\{\gamma^{l}_{k}\}_{l=1}^{g(N_{k})}$. 
Roughly speaking, there are $g(N_{k})$ disjoint area-minimizing discs $\{\Omega^{l}_{k}\}_{l}$ with $\partial  \Omega^{l}_{k}=\gamma^{l}_{k}$.
 Their existence is ensured by the works of Meeks and Yau (see \cite{YM, YM1} or Theorem 6.28 of \cite{CM1}) when the boundary  $\partial N_{k}$ is mean convex. 
 
 Let us explain their existence. We construct these discs by induction on $l$. 
 
 When $l=1$, there is an embedded area-minimizing disc $\Omega^{1}_{k}\subset N_{k}$ with boundary $\gamma^1_{k}$ (see \cite{YM, YM1} or Theorem 6.28 of \cite{CM1}).
 
 Suppose that there are $l$ disjointly embedded stable minimal discs $\{\Omega^{i}_{k}\}_{i=1}^{l}$ with $\partial \Omega^{i}_{k}=\gamma^{i}_k$. Our target is to construct a stable minimal surface $\Omega^{l+1}_{k}$ with boundary $\gamma^{l+1}_{k}$.
 
 Let us consider the Riemannian manifold $(T_{k, l}, g|_{T_{k,l}})$, where $T_{k, l}:=N_{k}\setminus \amalg_{i=1}^l \Omega^l_{k}$. It is a handlebody of genus $g(N_{k})-l$. For example, see the following figure.  
\begin{figure}[H]
\begin{center}
\begin{tikzpicture}
\node[draw,ellipse,minimum height=150pt, minimum width = 300pt,thick] (S) at (0,0){}; 
\draw(-2, 0) arc(200:340: 2 and 1.5);
\draw (-1.8, -0.3) arc(160: 20: 1.8 and 1);

\draw[dashed][red](0,-1) arc(90:-90: 0.4 and 0.8);
\draw[red](0, -1) arc(90: 270: 0.4 and 0.8);

 \begin{scope}
    \clip (0.4, -1.8) arc(0:360: 0.4 and 0.8);
   \draw[pattern=north west lines, pattern color=red] (-4,-4) rectangle (10,10);    
     \end{scope}

\draw (-4.5,-4)--(4.5, -4);

\draw(-4.5, -5.6)--(4.5, -5.6);
\draw[dashed][red](-4.5,-4) arc(90:-90: 0.4 and 0.8);
\draw[red](-4.5, -4) arc(90: 270: 0.4 and 0.8);

 \begin{scope}
    \clip (-4.1, -4.8) arc(0:360: 0.4 and 0.8);
   \draw[pattern=north west lines, pattern color=red] (-8,-7) rectangle (10,10);    
     \end{scope} 
     
    \begin{scope}
    \clip (4.9, -4.8) arc(0:360: 0.4 and 0.8);
   \draw[pattern=north west lines, pattern color=red] (-25,-20) rectangle (100,100);    
     \end{scope}

\draw[red](4.5,-4) arc(90:-90: 0.4 and 0.8);
\draw[red](4.5, -4) arc(90: 270: 0.4 and 0.8);

\draw[->] (0, -3) -- (0, -3.5);

\node(Os) at (0,-2) {$\Omega^{1}_{k}$};

\node(Os) at (0,2) {$(N_{k}, g_{k})$};
\node(Os) at (-0.7, -2) {$\gamma^{1}_{k}$};
\node(Os) at (0, -5) {$({T}_{k, 1}, g_{k}|_{T_{k,1}})$};

\node(Os) at (-4.4,-4.8) {${\Omega^{1}_{k}}^{-}$};
\node(Os) at (-5.2, -5) {${\gamma^{1}_{k}}^{-}$};
\node(Os) at (4.6,-4.8) {${\Omega^{1}_{k}}^{+}$};
\node(Os) at (3.7, -5) {${\gamma^{1}_{k}}^{+}$};
\label{1}
\end{tikzpicture}

\caption{}

\end{center}
\end{figure}

 The boundary of $(T_{k, l}, g|_{T_{k,l}})$ consists of $\partial N_{k}\setminus \amalg_{i=1}^{l}\gamma^{i}_{k}$ and some disjoint discs $\{{ \Omega_{k}^{i}}^{-}\}_{i=1}^{l}$ and $\{{\Omega_{k}^{i}}^{+}\}_{i=1}^l$. The two discs ${ \Omega_{k}^{i}}^{-}$ and ${\Omega_{k}^{i}}^{+}$ both come from the same minimal disc $\Omega^i_{k}$. Therefore,  the mean curvature of the boundary of $(T_{k,l}, g|_{T_{k,l}})$ is non-negative. (See Section 5.1) 

 In addition, $\{\gamma^{i}_{k}\}_{i>l}$ is a system of meridians of the handlebody $(T_{k, l}, g|_{T_{k,l}})$ Then, we use the result of Meeks and Yau to find an embedded stable minimal surface $\Omega^{l+1}_{k}\subset T_{k, l}$ with boundary $\gamma^{l+1}_{k}$. These discs $\{\Omega^{i}_{k}\}_{i=1}^{l+1}$ are disjoint in $N_{k}$.
  This finishes the inductive construction.

\vspace{2mm}

If $\partial N_{k}$ is not mean convex, we can deform the metric in a small neighborhood of it so that for the new metric, it becomes mean convex. As constructed above, each $\Omega^{l}_{k}$ is stable minimal for this new metric and for the original one away from a neighborhood of $\partial N_{k}$ (near $N_{k-1}$, for example). It is sufficient for our proof.

\vspace{2mm}

Define the lamination $\mathscr{L}_{k}:=\amalg_{l} \Omega^{l}_{k}$ (i.e. a disjoint union of embedded surfaces). We show that each lamination $\mathscr{L}_{k}$ intersects the compact set $R_{0}$ (from Corollary \ref{inter}).
According to Colding-Minicozzi's theory (see Appendix B of \cite{CM}), the sequence $\{\mathscr{L}_{k}\}_{k}$ sub-converges to a lamination $\mathscr{L}:=\cup_{t\in \Lambda}L_{t}$ in $(M, g)$. Note that each leaf $L_{t}$ is a complete (non-compact) stable minimal surface (see Theorem \ref{min-lim}).

As indicated above, since $(M, g)$ has positive scalar curvature and $\pi_{1}^\infty(M)$ is trivial, each leaf $L_{t}$ in $\mathscr{L}$ has the Vanishing property with respect to  $\{R_{k}\}_{k}$ (see Lemma \ref{van1} and Corollary \ref{trivial-van1}). 
Furthermore, the lamination $\mathscr{L}$ also satisfies the Vanishing property (see Corollary \ref{trivial-van2}).
That is to say, 

\vspace{2mm}

\emph{there exists a positive integer $k_{0}$ such that for any $k\geq k_{0}$ and any $t\in \Lambda$, any circle in $L_{t}\cap \partial R_{k}$ is null-homotopic in $\partial R_{k}$.}

\vspace{2mm}

The reason is described as follows.

We argue by contradiction. Suppose that there exists a sequence $\{k_{n}\}_{n}$ of increasing integers and a sequence $\{L_{t_{n}}\}$ of leaves in $\mathscr{L}$ satisfying that  $L_{t_{n}}\cap \partial R_{k_{n}}$ has at least one non-nullhomotopic circle in $\partial R_{k_{n}}$ for each $n$. 

The sequence $\{L_{t_{n}}\}$ smoothly subconverges to some  leaf in $\mathscr{L}$. For our convenience, we may assume that the sequence $\{L_{t_{n}}\}$ converges to the leaf $L_{t_{\infty}}$. 
The leaf $L_{t_{\infty}}$ satisfies the Vanishing property (see Lemma \ref{van1}).
 That is to say,  there is a positive integer $k(L_{t_{\infty}})$ such that for $k\geq k(L_{t_{\infty}})$, any circle in $\partial R_{k}\cap L_{t_{\infty}}$ is homotopically trivial  in $\partial R_{k}$. 

However, since $L_{t_{n}}\cap \partial R_{k_{n}}$ has some non-nullhomotopic circle in $\partial R_{k_{n}}$, we show that for $k_{n}>k(L_{t_{\infty}})$, $L_{t_{n}}\cap \partial R_{k(L_{t_\infty})}$ has a meridian of $ R_{k(L_{t_{\infty}})}$ (see Remark \ref{inj} and Corollary \ref{general one}). 
These meridians of $ R_{k(L_{t_{\infty}})}$ will converge to a meridian of $R_{k(L_{t_{\infty}})}$ which is contained in $L_{t_{\infty}}\cap \partial R_{k(L_{t_{\infty}})}$. This is in contradiction with the last paragraph. 

\vspace{2mm}

Let us explain how to deduce a contradiction from the Vanishing property of  $\mathscr{L}$.

We can show that if $N_{k}$ contains $R_{k_{0}}$ (for $k$ large enough), then $\mathscr{L}_{k}\cap \partial R_{k_{0}}$ contains at least one meridian(s) of $R_{k_{0}}$ (see Corollary \ref{inter}).
 Since $\mathscr{L}_{k}$ sub-converges to $\mathscr{L}$, then these meridians of $ R_{k_{0}}$ will sub-converge to a non-contractible circle in $\mathscr{L}\cap \partial R_{k_{0}}$. That is to say, some leaf $L_{t}$ in $\mathscr{L}$ contains this non-nullhomotopic circle in $\partial R_{k_{0}}$. This is in contradiction with the above fact (the Vanishing property of $\mathscr{L}$).

\vspace{2mm}

 \emph{1.4 The plan of this paper} 
 
 For the first part of the paper, we describe the topological properties of contractible $3$-manifolds. In Section 2, we recall some notions, such as simply-connectedness at infinity, fundamental group at infinity and handlebodies. In Section 3, we introduce meridian curves and meridian discs in a handlebody. In Section 4, we introduce two types of surgeries on handlebodies. Using these surgeries, we show the existence of an increasing family of handlebodies with good properties, called \emph{Property (H)}.
 
 In the second part, we treat minimal surfaces and related problems. In Section 5, we construct minimal laminations and consider the convergence theory of these laminations. In Section 6, we introduce the \emph{Vanishing property} and study its relation with the fundamental group at infinity. Their relationship  is clarified by Lemma \ref{van1}. 
 
 For the third part, we give the complete proof of Theorem \ref{A}. In Sections 7 and 8, our proof is  similar to the genus one case.  In Appendix C, we construct  a contractible $3$-manifold with non-trivial $\pi^{\infty}_{1}$. In addition, we prove that this manifold has no complete metric of positive scalar curvature.

\section{Background}

\subsection{Simply-connectedness at infinity and $\pi_{1}^{\infty}$}

\begin{definition}\label{sci}A topological space $M$ is  \emph{simply connected at infinity} if for any compact set $K\subset M$, there exists a  compact set $K'$ containing $K$ so that the induced map $\pi_{1}(M\setminus K')\rightarrow \pi_{1}(M\setminus K)$ is trivial. \end{definition}

The Poincar\'e conjecture (see \cite{BBBMP}, \cite{MT} and \cite{cao-zhu}) shows that any contractible $3$-manifold is irreducible (i.e. any embedded $2$-sphere in the $3$-manifold bounds a closed $3$-ball).
 A result of Stallings \cite{S} tells us that the only contractible and simply-connected at infinity $3$-manifold is $\mathbb{R}^{3}$.

\begin{remark}\label{ball}If a contractible $3$-manifold $M$ is not homeomorphic to $\mathbb{R}^{3}$, it is not simply-connected at infinity. That is to say, there is a compact set $K\subset M$ so that for any compact set $K'\subset M$ containing $K$, the induced map $\pi_{1}(M\setminus K')\rightarrow \pi_{1}(M\setminus K)$ is not trivial. We also know that the set $K$ is not contained in a $3$-ball in $M$. The reason is described below: 

\vspace{2mm}

If a closed $3$-ball $B$ (i.e. a closed set homeomorphic to a closed unit ball in $\mathbb{R}^{3}$) contains $K$,  Van-Kampen's Theorem shows that $\pi_{1}(M)\cong \pi_{1}(\overline{M\setminus B})\ast_{\pi_{1}(\partial B)}\pi_{1}(B)$. In addition, $\pi_{1}(B)$ and $\pi_{1}(\partial B)$ are both trivial. Therefore, $\pi_{1}(\overline{M\setminus B})\cong\pi_{1}(M)$ is trivial. That is to say, the map $\pi_{1}(M\setminus B)\rightarrow \pi_{1}(M\setminus K)$ is trivial. It leads to a contradiction.   
\end{remark}

\begin{definition}\label{fgi}The fundamental group at infinity $\pi_{1}^{\infty}$ of a path-connected space is the inverse limit of the fundamental groups of complements of compact subsets.\end{definition}

\begin{remark}\label{element} Let us consider a contractible $3$-manifold $M$. $\pi_{1}^{\infty}(M)$ is non-trivial if and only if  there is a compact set $K$ and a family $\{\gamma_{k}\}_{k}$ of circles  in $M\setminus K$ going to infinity with the property that  for each $k$

1) $\gamma_{k}$ is not null-homotopic in $M\setminus K$ and 

2) $\gamma_{k}$ is homotopic to $\gamma_{k+1}$ in $M\setminus K$.

Such a family of circles gives a non-trivial element in $\pi_{1}^{\infty}(M)$. 
\end{remark}

For any contractible $n$-manifold $M^{n}$, it is simply-connected at infinity if and only if $\pi_{1}^{\infty}(M^{n})$ is trivial, when $n\geq 4$ (See \cite{CWY}). However, this result is not true in dimension $3$. 

\begin{remark}
Any contractible genus one $3$-manifold $M$ is not homeomorphic to $\mathbb{R}^{3}$. It is not simply-connected at infinity but its fundamental group at infinity is trivial. 
The reason is as follows: 

\vspace{2mm}

Let $M$ be an increasing union of solid tori $\{N_k\}^\infty_{k=1}$. Lemma 2.10 in \cite{W1} shows that  the  induced maps $\pi_1(\p N_k)\rightarrow \pi_1(N_k\setminus N_0)$ and $\pi_1(\p N_k)\rightarrow\pi_1(\overline{M\setminus N_k})$ are both injective for $k>0$. Namely, the family $\{N_k\}$ is excellent (see Section 2 of \cite{Myers}).

A result (see Lemma 4.1 of \cite{Myers} on Page 33) shows that  if a closed curve $\gamma\subset M\setminus N_{k+1}$ is homotopic to $\gamma'\subset N_{k}\setminus N_m$ in $M\setminus N_0$, then $\gamma$ is contractible in $M\setminus N_0$ where  $0<m<k$ . 

\vspace{2mm}

Remark \ref{element} shows that an element in $\pi^{\infty}_1(M)$ gives a family of circles $\{\gamma_i\}^\infty_{i=1}$ going to infinity and a compact set $K$. We may assume that $K$ is equal to $N_0$.  Each $\gamma_i$ is homotopic to $\gamma_1$ in $M\setminus N_0$.

Let $\gamma_1$ be a subset of some $N_k\setminus N_0$. We choose $i$ large enough so that $\gamma_i$ is a subset of $M\setminus N_{k+1}$. We use the above statement to have that $\gamma_i$ is contractible in $M\setminus N_0$. Namely, $\pi^{\infty}_1(M)$ is trivial.

\end{remark}

\subsection{Handlebodies}
\begin{definition}[Page 46, \cite{Rol}] A closed handlebody is any space obtained from the closed $3$-ball $D^{3}$ ($0$-handle) by attaching $g$ distinct copies of $D^{2}\times[-1, 1]$ ($1$-handle) with the homeomorphisms identifying the $2g$ discs $D^{2}\times\{\pm1\}$ to $2g$ disjoint $2$-disks on $\partial D^{3}$, all to be done in such a way that the resulting $3$-manifold is orientable. The integer $g$ is called the genus of the handlebody. 
\end{definition}
Let us remark that a handlebody of genus $g$ is homeomorphic to a boundary connected sum of $g$ solid tori. Therefore, its boundary is a compact surface of genus $g$. (See Page 46 in \cite{Rol})

From a result of McMillan \cite{Mc1} and the Poincar\'e conjecture (see \cite{BBBMP, MT, cao-zhu}), we know that: 
\begin{theorem}\label{handle} \textnormal{(\cite{Mc1}, Theorem 1 on Page 511)}Any contractible $3$-manifold can be written as an ascending union of handlebodies . 
\end{theorem}

\begin{remark}\label{stru}
Let us consider a contractible $3$-manifold $M$. If it is not homeomorphic to $\mathbb{R}^{3}$, it can written as an increasing family of handlebodies $\{N_{k}\}$ satisfying  that
\begin{itemize}
\item  $N_{k}$ is homotopically trivial in $N_{k+1}$ for each $k$;
\item  none of the $N_{k}$ is contained in a $3$-ball (from Remark \ref{ball}).
\end{itemize}
\end{remark}

%\begin{definition}\label{bounded} A $3$-manifold $M$ is of bounded genus (at most $g$) if there is a constant $g$ so that it is an increasing union of open handlebodies whose genus are not greater than $g$. \end{definition}
%
%%\begin{rem} If there are infinitely many handlebodies of genus zero (i.e. $3$-balls), the $3$-manifold is $\mathbb{R}^{3}$. 
%%\end{rem}
%For example, $\mathbb{R}^{3}$ and genus one $3$-manifolds are both of bounded genus. 
%
\section{Meridians} In this part, we consider  a closed handlebody $N$.

\begin{definition}\label{mer}
An embedded circle $\gamma\subset \partial N$ is called a \emph{meridian} if $\gamma$ is null-homotopic in ${N}$, but not contractible in $\partial N$. 

An embedded closed disc $(D, \partial D)\subset ({N}, \partial N)$ is called a \emph{meridian disc} if its boundary is a meridian of $N$.

The disc $D$ is  a \emph{separating meridian disc}, if $N\setminus D$ is not connected. Its boundary is called  a \emph{separating meridian}. 

The disc $D$ is  a \emph{non-separating disc}, if $N\setminus D$ is connected. Its boundary is called a \emph{non-separating meridian}.

\end{definition}

\begin{rem} Let us consider a meridian $\gamma$ of $N$. If $\gamma$ is a separating meridian, it cuts $\partial N$ into two components. The class $[\gamma]$ is equal to zero in $H_{1}(\partial N)$. 

If  $\gamma$ is a non-separating meridian, then $\partial N\setminus \gamma$ is connected. The class $[\gamma]$ is a non-trivial element in $H_{1}(\partial N)$. 
\end{rem}

\subsection{Effective meridians}
Consider two closed handlebodies $N'$ and $N$ with $N'\subset \text{Int}~N$. 

\begin{definition}\label{eff}
A meridian $\gamma$ of $N$ is called \emph{an effective meridian relative to $N'$} if any meridian disc with boundary $\gamma$ intersects the core of $N'$ (i.e. a  deformation retraction of $N'$), where the core of $N'$ is a $1$-dimensional CW complex.\par

The handbody $N$ is called \emph{an effective handlebody relative to $N'$}, if any meridian of $N$ is an effective meridian relative to $N'$.
\end{definition}
Note that if $N'$ is contained in a $3$-ball $B\subset \text{Int}~N$, there is no effective meridian relative to $N'$.

In the following, we will repeatedly use the Loop lemma. 

\begin{lemma}\label{loop} \textnormal{(\cite{HA}, Theorem 3.1 on Page 54)}Let $M$ be an orientable  $3$-manifold with boundary $\partial M$, not necessarily compact . If there is a map $f: (D^{2}, \partial D^{2})\rightarrow (M, \partial M)$ with the property that $f|_{\partial D^{2}}$ is not null-homotopic in $\partial M$,  then there is an embedding $h$ with the same property. \end{lemma}

\begin{remark}\label{Int} We may assume that $h(\text{Int}~D^{2})\subset \text{Int} ~ M$. The reason is described below: 

Let us consider a 1-sided open neighborhood $ M_{\epsilon} \cong \partial M \times [0, \epsilon)$ of $\partial M$ in $M$.  Shrinking the image of $f$ into $M(\epsilon):=M\setminus M_{\epsilon}$, we find a map $f_{\epsilon}:(D^{2}, \partial D^{2})\rightarrow (M(\epsilon), \partial M(\epsilon))$ with the property that $f_{\epsilon}(\partial D^{2})$ is not nullhomotopic in $\partial M(\epsilon)$. We use Lemma \ref{loop} to find an embedding $h_{\epsilon}$ with the same property. Its image stays in $(M(\epsilon), \partial M(\epsilon))$. Therefore, the image of $h_{\epsilon}$ is contained in $\text{Int}~M$.

In addition, there is an embedded circle $\gamma\subset \partial M$ which is homotopic to $h_{\epsilon}(\partial D^{2})$ in $\overline{M}_{\epsilon}$. There is  an embedded annulus $A_{\epsilon}\subset \overline{M_{\epsilon}}$ joining $\gamma$ and $h_{\epsilon}(\partial D^{2})$. We find a map $h:(D^{2}, \partial D^{2})\rightarrow (M, \partial M)$ so that its image is an embedded disc (i.e. the union of $A_{\epsilon}$ and the image of $h_{\epsilon}$). It has the same property as $f$ and $h(\text{Int}~D^{2})\subset \text{Int}~M$.  
\end{remark}

\begin{lemma}\label{injective} Let $N'$ and $N$ be two closed handlebodies with $N'\subset \text{Int}~N$. The handlebody $N$ is  an effective handlebody relative to $N'$ if and only if  the map $\pi_{1}(\partial N) \rightarrow\pi_{1}(\overline{N\setminus N'})$ is injective.
\end{lemma}
\begin{proof}  If $N$ is not an effective handlebody relative to $N'$, there is a meridian disc $(D, \partial D)\subset ({N}, \partial N)$ with $D\cap N'=\emptyset$. Therefore, the map $\pi_{1}(\partial N)\rightarrow\pi_{1}(\overline{N\setminus N'})$ is not injective.

If the map $\pi_{1}(\partial N)\rightarrow \pi_{1}(\overline{N\setminus N'})$ is not injective,  we use Lemma \ref{loop} to find an embedded disc $(D', \partial D')\subset (\overline{N\setminus N'}, \partial N)$ whose boundary is not contractible in $\partial N$. As in Remark \ref{Int}, we may assume that $\text{Int} ~D' \subset \text{Int} ~(N\setminus N')$. We see that $D'$ is a meridian disc with $D'\cap N'=\emptyset$. Therefore, $N$ is not an effective handlebody relative to $N'$. 
This finishes the proof. \end{proof}

We now introduce some notations about circles in a disc.
\begin{definition}(See Definition 2.11 of \cite{W1})\label{outmost}Let $C:=\{c_{i}\}_{i\in I}$ be a finite set of pairwise disjoint circles in the disc $D^{2}$ and $D_{i}\subset D^{2}$ the unique disc with boundary $c_{i}$. Consider the set $\{D_{i}\}_{i\in I}$ and define a partially ordered relation induced by the inclusion. For each maximal element $D_{j}$ in $(\{D_{i}\}_{i\in I}, \subset)$, its boundary $c_{j}$ is defined as a \emph{maximal circle} in $C$. For each minimal element $D_{j}$, its boundary $c_{j}$ is called a \emph{minimal circle} in $C$.
\end{definition}

\begin{lemma}\label{one} Let $N'$ and $N$ be two closed handlebodies satisfying 1) $N'\subset \text{Int}~N$ and 2) $\pi_{1}(\partial N')\rightarrow \pi_{1}(\overline{N\setminus N'})$ is injective. If $N$ is an effective handlebody relative to $N'$, then any meridian disc $(D, \partial D)\subset ({N}, \partial N)$ contains a meridian of $N'$. 
\end{lemma}

The proof is the same as the proof of Lemma 2.12 in \cite{W1}.
\begin{proof} 
Suppose that the closed meridian disc $D$ intersects $\partial N'$ transversally where $\gamma:=\partial D$ is a meridian of $N$. The intersection $D\cap \partial N'$ is a disjoint union of  circles $\{c_{i}\}_{i\in I}$. Each $c_{i}$ bounds a unique closed disc $D_{i}\subset \text{Int}~D$.
 
  Consider the set $C^{non}:=\{c_{i}~|~ c_{i}~ \text{is homotopically trivial in }~ \partial N'\}$ and the set $C^{max}=\{c_{i}|~c_{i}$ is a maximal circle in $\{c_{i}\}_{i\in I}\}$.

We will show that $C^{non}$ is nonempty and a minimal circle in $C^{non}$ is a desired meridian.

 Suppose the contrary that $C^{non}$ is empty. 
Hence, each $c_{i}\in C^{max}$ is contractible in $\partial N'$ and bounds a disc $D'_{i} \subset \partial N'$.
Consider the immersed disc $$\hat{D}:=(D\setminus\cup_{c_{i}\in C^{max}}D_{i})\cup(\cup_{c_{i}\in C^{max}}D'_{i})$$ with boundary $\gamma$. Since $\hat{D}\cap \text{Int}~N'=\emptyset$, we see that $\gamma$ is contractible in $\overline{N\setminus N'}$.

However, Lemma \ref{injective} shows that the map $\pi_{1}(\partial N)\rightarrow \pi_{1}(\overline{N\setminus N'})$ is injective. That is to say, the circle $\gamma$ is null-homotopic in $\partial N$. This is in contradiction with our hypothesis that $\gamma$ is non-trivial in $\pi_{1}(\partial N)$. We conclude that $C^{non}\neq\emptyset$.

\vspace{2mm}

In the following, we will prove that each minimal circle $c_{j}$ in $C^{non}$ is a required meridian. From Definition \ref{mer}, it is sufficient to show that $c_{j}$ is homotopically trivial in $N'$.  Our strategy is to  construct an immersed disc $\hat{D}_{j}\subset {N'}$ with boundary $c_{j}$.

Let $C_{j}:=\{c_{i}~| c_{i}\subset \text{Int} ~D_{j}$ for $i\in I\}$ and  $C_{j}^{max}$ be the set of maximal circles in $C_{j}$. We now have two cases:  $C_{j}=\emptyset$ or $C_{j}\neq \emptyset$. 

\textbf{Case I:} If $C_{j}$ is empty, we consider the set $Z:=\text{Int}~D_{j}$ and define the disc $\hat{D}_{j}$ as $\text{Int}~D_{j}$. 

\textbf{Case II:} If $C_{j}$ is not empty, then $C^{max}_{j}$ is also nonempty. 
From the minimality of $c_{j}$ in $C^{non}$, each $c_{i}\in C^{max}_{j}$ is null-homotopic in $\partial N'$ and bounds a disc $D''_{i}\subset\partial N'$. 

Define  the set $Z:=\text{Int}~D_{j}\setminus \cup_{c_{i}\in C^{max}_{j}}{D}_{i}$ and the new disc $\hat{D}_{j}:=Z\cup(\cup_{c_{i}\in C^{max}_{j}} D''_{i})$ with boundary $c_{j}$. 

\vspace{2mm}

Let us explain why $\hat{D}_{j}$ is contained in ${N'}$. In any case, $\partial N'$ cuts $N$ into two connected components, $N\setminus N'$ and $\text{Int}~N'$. The set $Z$ is one of these components of $\text{Int}~D_{j}\setminus \partial N'$. Therefore, it must be contained in $\text{Int}~N'$ or $N\setminus {N'}$.

If $Z$ is in $N\setminus {N'}$, then the disc $\hat{D}_{j}$  is contained in $\overline{N\setminus N'}$. Namely, $c_{j}$ is contractible in $\overline{N\setminus N'}$. However, since the induced map $\pi_{1}(\partial N')\rightarrow\pi_{1}(\overline{N\setminus N'})$ is injective, then $c_{j}$ is homotopically trivial in $\partial N'$. This contradicts  the choice of $c_{j}\in C^{non}$. We conclude that $Z$ is contained in $\text{Int}~N'$.

Therefore, $\hat{D}_{j}$ is contained in ${N'}$. That is to say, $c_{j}$ is null-homotopic in ${N'}$. However, it is non-trivial in $\pi_{1}(\partial N')$. From Definition \ref{mer}, we conclude that $c_{j}\subset D$ is a meridian of $N'$. This finishes the proof. 
\end{proof}

As a consequence, we have 

\begin{corollary}\label{general one} Let $N'$ and $N$ be two closed handlebodies in a contractible $3$-manifold $M$ satisfying that 1) $N' \subset \text{Int}~N$ and 2) the map $\pi_{1}(\partial N')\rightarrow \pi_{1}(\overline{M\setminus N'})$ is injective.  If an embedded circle $\gamma\subset \partial N$ is not null-homotopic in $\overline{M\setminus N'}$, then any embedded disc $D\subset M$ with boundary $\gamma$ contains a meridian of $N'$. 
\end{corollary}
The proof is the same as Lemma \ref{one}. 
\subsection{Non-sparating meridians }

\begin{lemma}\label{sys}
For a closed handlebody $N$ of genus $g$, there are $g$ disjoint non-separating meridians $\{\gamma^{l}\}^{g}_{l=1}$ so that  $N\setminus \amalg_{l} N_{\epsilon_{l}} (D_{l})$ is a closed 3-ball, where  $D_{l}$ is a closed meridian disc with boundary $\gamma^{l}$ and $N_{\epsilon_{l}}(D_{l})$ is an open neighborhood of $D_{l}$ in $N$ with small radius $\epsilon_{l}$ . 
\end{lemma}
The set of these meridians $\{\gamma^{l}\}_{l=1}^{g}$ is called \emph{a system of the handlebody $N$ of genus $g$}. It is not unique. 
\begin{proof}Pick any non-separating meridian $\gamma^{1}$ of $N$. We use Lemma \ref{loop} to find an embedded disc $D_{1}\subset N$.

As in Remark \ref{Int}, we may assume that $\text{Int}~ D_{1}\subset \text{Int} ~N$.
 The set $N_{1}:=N\setminus N_{\epsilon}(D_{1})$ is a closed handlebody of genus $g-1$, where $N_{\epsilon_{1}}(D_{1})$ is the open tubular neighborhood of $D_{1}$ in $N$ with small radius $\epsilon_{1}$. In particular,  the map $\pi_{1}(\partial N\cap \partial N_{1})\rightarrow \pi_{1}(\partial N_{1})$ is surjective.  
 
Choose a non-separating meridian $\gamma^{2}\subset \partial N\cap \partial N_{1}$ of $ N_{1}$. By Lemma \ref{loop}, there exists a meridian disc $D_{2}$ of $N_{1}=N\setminus N_{\epsilon_{1}}(D_{1})$. The set $N_{2}:=N\setminus N_{\epsilon_{1}} (D_{1})\amalg N_{\epsilon_{2}}(D_{2})$ is a closed handlebody of genus $g-2$, where $N_{\epsilon_{2}}(D_{2})$ is an open tubular neighborhood of $D_{2}$ in $N$. 

We repeat this process $g-2$ times and obtain $g$ disjointly embedded discs $\{D_{l}\}$ so that  $N\setminus \amalg_{l} N_{\epsilon_{l}}(D_{l})$ is a handlebody of genus zero (a 3-ball).  The boundaries $\{\gamma^{l}\}^{g}_{l=1}$ of these discs are $g$ distinct meridians which are the required candidates in the assertion.   \end{proof}

\begin{corollary}\label{inter} Let $N\subset M$, $\{\gamma^{l}\}$ and $\{D_{l}\}$ be as in Lemma \ref{sys}, where $M$ is a 3-manifold without boundary. If $R\subset \text{Int}~N$ is a closed handlebody satisfying that 1) it is not contained in a 3-ball in $M$; 2) $\pi_{1}(\partial R)\rightarrow \pi_{1}(\overline{M\setminus R})$ is injective, then $\partial R\cap \amalg_{l} D_{l}$ contains at least a meridian of $R$.
\end{corollary}

The proof is also similar to the proof of Lemma 2.12 in \cite{W1}. 

\begin{proof}
We may assume that $\partial R$ intersects $\amalg_{l}D_{l}$ transversally. The intersection $\partial R\cap \amalg D_{l}:=\{\gamma\}_{\gamma\in C}$ has finitely many components. 
Let us consider the set $C^{non}:=\{\gamma\in C$ is not homotopically trivial in $\partial R\}$. 

\vspace{2mm}
\noindent\textbf{Claim: } $C^{non}$ is nonempty. \par

We argue by contradiction. Suppose that $C^{non}$ is empty.  We see that any circle in $D_{l}\cap \partial R$ is  null-homotopic in $\partial R$. As in the proof of Lemma \ref{one},  we get a new disc in $\overline{N\setminus R}$ with boundary $\gamma^{l}$. Therefore, each $\gamma^{l}$ is null-homotopic in $\overline{N\setminus R}$. 

We use Lemma \ref{loop} to find a meridian disc $D'_{1}\subset \overline{N\setminus R}$ with boundary $\gamma^{1}$. As in Remark \ref{Int}, we may assume that $\text{Int}~D'_{1}\subset \text{Int}~\overline{N\setminus R}$ (or $D'_{1}\subset N\setminus R$).
Choose an open  tubular neighborhood $N_{\epsilon'_{1}}(D'_{1})$ of $D'_1$ in $N\setminus R$ with small radius $\epsilon'_{1}$. The set $N'_{1}:=N\setminus N_{\epsilon'_{1}}(D'_{1})$ is a closed handlebody of genus $g-1$ containing $R$.

  In addition, for $l>1$, $\gamma^{l}$ is a non-separating meridian of $N'_{1}$ but null-homotopic in $N\setminus (N_{\epsilon'_{1}}(D'_{1})\amalg R)$.

Repeating this process $g-1$ times, we obtain $g$ embedded discs $\{D'_{l}\}_{l=1}^g$ so that 

1) $R\cap \amalg_{l}N_{\epsilon'_{l}}(D'_{l})=\emptyset$;  

2) The handlebody $N\setminus \amalg_{l} N_{\epsilon'_{l}} (D'_{l})$ is of genus zero (a closed $3$-ball),
 
\noindent where $N_{\epsilon'_{l}}(D'_{l})$ is an open tubular neighborhood of $D'_{l}$ in $N$ with small radius $\epsilon'_{l}$.  

Therefore, $R$ is contained in the $3$-ball $N\setminus \amalg_{l} N_{\epsilon'_{l}}(D'_{l})$. This contradicts our hypothesis. It finishes the proof of the claim.

As in the proof of Lemma \ref{one}, we use the condition 2) to show that each minimal circle in $C^{non}$ is a required meridian.
\end{proof}

\section{Effective Handlebodies}

\subsection{Surgeries}
Consider  two closed handlebodies $N'$ and $N$ in a contractible $3$-manifold $M$ with $N'\subset \text{Int}~N$. We  introduce two types of surgeries on handlebodies: 

\vspace{2mm}
\noindent \textbf{Type I}: If there exists a meridian disc $D\subset N\setminus N'$ of $N$, then  we consider an open tubular neighborhood $N_{\epsilon}(D)\subset N\setminus N'$ of $D$. We then have two cases:

Case (1): If $D$ is a separating meridian disc, $N\setminus N_{\epsilon}(D)$ has two components. The closed handlebody $W_{1}$ is defined as the component containing $N'$;

Case(2): If $D$ is a non-separating meridian disc, $N\setminus N_{\epsilon}(D)$ is connected. The closed handlebody $W_{1}$ is defined by $N\setminus N_{\epsilon}(D)$. 

\vspace{2mm}
\noindent \textbf{Type II}: Let $\gamma$ be a non-contractible circle in $\partial N$ and  $D_{1}\subset \overline{M\setminus N}$ an embedded disc  with   boundary $\gamma\subset \partial N$. We may assume that $\text{Int}~D_{1}\subset M\setminus N$. 

Case (1): If $\gamma$ is not homotopically trivial in $N$, we consider a closed tubular neighborhood $\overline{N_{\epsilon_{1}}(D_{1})}$ of $D_{1}$ in $\overline{M\setminus N}$. Define a new handlebody $W_{2}$ as $N\cup \overline{N_{\epsilon_{1}}(D_{1})}$. 

Case (2): If $\gamma$ is null-homotopic in $N$, it is a meridian of $N$. Consider a meridian disc $D_2$ of $N$ with boundary $\gamma$ and the embedded sphere $D_1\cup_{\gamma}D_2$.  The Poincar\'e conjecture (see \cite{BBBMP}, \cite{MT} and \cite{cao-zhu}) implies that  the sphere $D_1\cup_{\gamma}D_2$ bounds a $3$-ball $B$. 

We can conclude that $D_2$ is a separating meridian disc. (Otherwise, we find a circle $\gamma_0\subset N$ so that the intersection number of the sphere $D_1\cup_\gamma D_2$ and $\gamma_0$ is $\pm 1$. However, since any sphere in $M$ contractible, their intersection must be zero, a contradiction.)

Therefore, $B$ contains one of components of $N\setminus D_2$ and the union $B\cup N$ is also a handlebody. The new handlebody  $W_2$ is defined by $B\cup N$.

\begin{remark}\label{genus}
 For $i=1,2$,  the genus $g(\partial W_{i})$ of  $\partial W_{i}$ is less than $g(\partial N)$. In addition, $\partial W_{i}$ is a union of $\partial W_{i}\cap \partial N$ and some disjoint discs. This shows that the map $\pi_{1}(\partial W_{i}\cap \partial N)\rightarrow \pi_{1}(\partial W_{i})$  is surjective. 
 \end{remark}
 
 \begin{lemma}\label{contractible} If $N'$ is homotopically trivial in $N$, then $N'$ is also homotopically trivial in $W_{i}$ for each $i$, where $W_{i}$ is obtained from the above surgeries.  \end{lemma}

\begin{proof} For the type II surgery,  $N$ is contained in $W_{2}$. Therefore,  $N'$ is homotopically trivial in $W_{2}$.

 For the type I surgery, it is sufficient to show  that any circle $c \subset N'$  bounds some (immersed) disc $\hat{D}'\subset W_{1}$.  
 
 The closed curve $c$ bounds an immersed disc $D'\subset \text{Int}~N$. We will construct the required disc $\hat{D}'\subset W_{1}$ from $D'$.

  We may assume that $D'$ intersects $D^{-}\amalg D^{+}:=\text{Int}~N\cap \partial N_{\epsilon}(D')$ transversally.  Each component $c_{i}$ of $D'\cap (D^{+}\amalg D^{-})$ is a circle in $D'$ and bounds a closed sub-disc $D'_{i}\subset D'$. 
  
  Since $D^{+}$ and $ D^{-}$ are two disjoint discs, each $c_{i}$ is contractible in $D^{+}\amalg D^{-}$. It also bounds a disc $D''_{i}\subset D^{+}\amalg D^{-}$. Let $C^{max}$ be the set of the maximal circles of  $\{c_{i}\}_{i\in I}$ in $D'$. We construct a disc $$\hat{D}':=D'\setminus \cup_{c_{i}\in C^{max}}D'_{i}\cup(\cup_{c_{i}\in C^{max}}D''_{i})$$ with boundary $c$. It stays in $\overline{N\setminus N_{\epsilon}(D')}$.  That is to say, $c$ is contractible in $W_{1}$. Therefore, $N'$ is homotopically trivial in $W_{1}$. 
  \end{proof}
  
\subsection{Existence of effective handlebodies}
In the following, let us consider a contractible $3$-manifold $M$. 
\begin{theorem}\label{modify}Let $N'$ and $N$ be two closed handlebodies in $M$ satisfying that 1) $N'\subset \text{Int}~N$ and 2) $N'$ is homotopically trivial in $N$.Then there exists a closed handlebody $R\subset M$ containing $N'$ satisfying that 

\begin{enumerate}
\item the map $\pi_{1}(\partial R)\rightarrow \pi_{1}(\overline{R\setminus N'})$ is injective;
\item the map $\pi_{1}(\partial R)\rightarrow \pi_{1}(\overline{M\setminus R})$ is injective;
\item $N'$ is homotopically trivial in $R$;
\item $\partial R$ is a  union of $\partial R\cap \partial N$ and some disjoint discs. 
\end{enumerate}
\end{theorem}

\begin{rem}
From Lemma \ref{injective}, $R$ is an effective handlebody relative to $N'$.
\end{rem} 

\begin{proof}  Suppose that either the map $i_{1}: \pi_{1}(\partial N)\rightarrow \pi_{1}(\overline{N\setminus N'})$ is not injective or the map $i_{2}:\pi_{1}(\partial N)\rightarrow \pi_{1}(\overline{M\setminus N})$ is not injective. (If these two maps are both injective, $R$ is defined as  $N$.) 

If $i_{1}$ is not injective, Lemma \ref{loop} shows that there exists a meridian disc $D_{1}$ of $N$ with $D_{1}\cap N'=\emptyset$. We do the type I surgery on $N$ with the disc $D_{1}$ to obtain a new handlebody $W$. 

If $i_{2}$ is not injective, we use Lemma \ref{loop} to find an embedded circle $\gamma\subset \partial N$ and an embedded disc $D_{2}\subset \overline{M\setminus N}$ ($\text{Int} ~D_{2}\subset M\setminus N$) where $\gamma=\partial D_{2}$ is not null-homotopic in $\partial N$. We do the type II surgery with the disc $D_{2}$ to get a new handlebody $W$. 

In any case, we have that $g(\partial W)<g(\partial N)$. The boundary $\partial W$ is a union of $\partial W\cap \partial N$ and some disjoint discs $\{D'_{i}\}_{i}$. Therefore, $\pi_{1}(\partial W\cap \partial N)\rightarrow \pi_{1}(\partial W)$ is surjective. In addition,  Lemma \ref{contractible} implies that $N'$ is contractible in $W$.

When picking a  circle $\gamma'\subset \partial W$ which is not null-homotopic in $\partial W$, we may assume that $\gamma'$ is an embedded circle in $\partial W\cap \partial N$. Therefore, when repeating these two types of surgeries, we may assume that the new surgeries are operated away from these disjoint discs $\{D'_{i}\}$.  

Iterate this process until we find a handlebody $R$ satisfying (1) and (2).  At each step, the genus of the handlebody obtained from the surgery is less than the original one. Therefore, this process stops in no more than $g(N)$ steps.

As above, $N'$ is homotopically trivial in $R$ and $\partial R$ is a union of $\partial R\cap \partial N$ and some disjoint discs. 
\end{proof}

\begin{rem} If $N'$ is not contained in a $3$-ball in $M$, then the genus of $R$ is greater than zero. 
\end{rem}

\begin{lemma}\label{upper} Let $R\subset M$ be a closed effective handlebody relative to the closed handlebody $N'\subset \text{Int}~R$ satisfying $\pi_{1}(\partial R)\rightarrow \pi_{1}(\overline{M\setminus R})$ is injective. If a closed handlebody $N$ is an effective handlebody relative to $R\subset \text{Int}~N$, then $N$ is an effective handlebody relative to $N'$.  
\end{lemma}

\begin{proof} From Lemma \ref{injective}, it is sufficient to show that the map $\pi_{1}(\partial N)\rightarrow \pi_{1}(\overline{N\setminus N'})$ is injective. 

We use Lemma \ref{injective} to show that the induced map $\pi_{1}(\partial R)\rightarrow \pi_{1}(\overline{R\setminus N'})$  is injective. 
Since $\pi_{1}(\partial R)\rightarrow \pi_{1}(\overline{M\setminus R})$ is injective, then the map $\pi_{1}(\partial R)\rightarrow \pi_{1}(\overline{N\setminus R})$ is also injective. 

Van Kampen's theorem gives an isomorphism between $\pi_{1}(\overline{N\setminus N'})$ and $\pi_{1}(\overline{N\setminus R})\ast_{\pi_{1}(\partial R)}\pi_{1}(\overline{R\setminus N'})$. A classical result (See [Theorem 11.67, Page 404] in \cite{Rot}) shows that the induced map $\pi_{1}(\overline{N\setminus R})\rightarrow \pi_{1}(\overline{N\setminus N'})$ is injective. 

 Lemma \ref{injective} shows that the map $\pi_{1}(\partial N)\rightarrow \pi_{1}(\overline{N\setminus R})$ is injective. Therefore, the composition $\pi_{1}(\partial N)\rightarrow \pi_{1}(\overline{N\setminus R})\rightarrow \pi_{1}(\overline{N\setminus N'})$ is also injective. This finishes the proof. 
\end{proof}

 \subsection{Property $(H)$} In the following, let us consider a contractible $3$-manifold $M$ which is not homeomorphic to $\mathbb{R}^{3}$. 
 
 By Theorem \ref{handle}, $M$ can be written as an ascending union of handlebodies $\{N_{k}\}^{\infty}_{k=0}$. 
 Each $N_{k}$ is homotopically trivial in $N_{k+1}$. 
 As in Section 2, we can assume that $N_{0}$ is not contained in a $3$-ball in $M$ (because $M$ is not homeomorphic to $\mathbb{R}^{3}$). (See Remark \ref{ball})
 
In the genus one case, the family $\{N_{k}\}$ has several good properties. 
For example, each $N_{k}$ is an effective handlebody relative to $N_{0}$ and the map $\pi_{1}(\partial N_{k})\rightarrow \pi_{1}(\overline{M\setminus N_{k}})$ is injective (See Lemma 2.10 of \cite{W1}). 
These properties are necessary and crucial in our proof.  In general, the family $\{N_{k}\}$ may not have these properties. 
To overcome this difficulty, we introduce a topological property, called \emph{Property (H)}.  

\begin{definition} \label{P H}A family $\{R_{k}\}_{k}$ of handlebodies in a contractible $3$-manifold $M:=\cup_{k}N_{k}$ is said to have \emph{Property $(H)$} if  
\begin{enumerate}
\item the map $\pi_{1}(\partial R_{k})\rightarrow \pi_{1}(\overline{R_{k}\setminus R_{0}})$ is injective for $k>0$;
\item the map $\pi_{1}(\partial R_{k})\rightarrow \pi_{1}(\overline{M\setminus R_{k}})$ is injective for $k\geq 0$;
\item each $R_{k}$ is contractible in $R_{k+1}$ but not contained in a $3$-ball in $M$ ;
\item there exists a sequence of increasing integers $\{j_{k}\}_{k}$, such that $\pi_{1}(\partial R_{k}\cap \partial N_{j_{k}})\rightarrow \pi_{1}(\partial R_{k})$ is surjective.
\end{enumerate}
where $\{N_{k}\}$ is  a fixed family of handlebodies assumed as in Section 2.
 \end{definition}

For example, in a contractible genus one $3$-manifold $M=\cup_k N_{k}$, the family $\{N_{k}\}$ satisfies \emph{Property (H)} (see Lemma 2.11  of \cite{W1}).

 In the following, we will prove that if a contractible $3$-manifold $M$ is not homeomorphic to $\mathbb{R}^3$, there is a family of handlebodies with Property (H). However, such a family is not unique.

 \begin{theorem}\label{H} If a contractible $3$-manifold $M:=\cup_{k}N_{k}$ (as above) is not homeomorphic to $\mathbb{R}^{3}$, then there is a family $\{R_k\}_{k}$ of handlebodies with  Property (H). Namely,  the family  $\{R_k\}_k$ satisfies the following: 
 \begin{enumerate}
\item the map $\pi_{1}(\partial R_{k})\rightarrow \pi_{1}(\overline{R_{k}\setminus R_{0}})$ is injective for $k>0$;
\item the map $\pi_{1}(\partial R_{k})\rightarrow \pi_{1}(\overline{M\setminus R_{k}})$ is injective for $k\geq 0$;
\item each $R_{k}$ is contractible in $R_{k+1}$ but not contained in a $3$-ball in $M$;
\item there exists a sequence of increasing integers $\{j_{k}\}_{k}$, such that $\pi_{1}(\partial R_{k}\cap \partial N_{j_{k}})\rightarrow \pi_{1}(\partial R_{k})$ is surjective.
\end{enumerate}
\end{theorem}

\begin{remark} \label{inj}~~~~~~~~~~~~~~~~~~~~~~~~~~~~~~
\begin{itemize}
\item The union $\cup_{k}R_{k}$ may be  not equal to $M$.
\item  For $k>0$,  Van-Kampen's Theorem gives an isomorphism between $\pi_{1}(\overline{M\setminus R_{0}})$ and $\pi_{1}(\overline{M\setminus R_{k}})\ast_{\pi_{1}(\partial R_{k})}\pi_{1}(\overline{R_{k}\setminus R_{0}})$. By (1) and (2), we use [Theorem 11.67, Page 404] in \cite{Rot} to show that the map $\pi_{1}(\partial R_{k})\rightarrow \pi_{1}(\overline{M\setminus R_{0}})$ is injective. 
\item As (4) in Theorem \ref{modify}, $\partial R_{k}$ is the union of $\partial R_{k}\cap \partial N_{j_{k}}$ and disjoint discs. 
\end{itemize}
\end{remark}

\begin{proof} First, we construct $R_{0}$. We repeatedly apply the Type II surgery to $N_{0}$, until we  find a handlebody $R_{0}$ containing $N_{0}$ so that $\pi_{1}(\partial R_{0})\rightarrow \pi_{1}(\overline{M\setminus R_{0}})$ is injective. 

From Remark \ref{genus}, we have that, at each step, the genus of the handlebody obtained from the surgery is less than the original one. Therefore, this process stops in no more than $g(N_{0})$ steps.

In addition, since $N_{0}$ is not contained in a $3$-ball in $M$,  then $R_{0}$ has the same property. 

\vspace{2mm}

In the following, we  construct the sequence $\{R_{k}\}_{k}$ inductively.

When $k$ is equal to $1$, we pick a handlebody $N_{j_{1}}$ containing $R_{0}$ satisfying that $R_{0}$ is homotopically trivial in $N_{j_{1}}$. Its existence is ensured by the following fact: 

Because $R_{0}$ is compact, there is some handlebody $N_{j_{1}-1}$ containing $R_{0}$. Since $N_{j_{1}-1}$ is homotopically trivial in $N_{j_{1}}$, $R_{0}$ is contained in $N_{j_{1}}$ and contractible in $N_{j_{1}}$.

\vspace{2mm}

By Theorem \ref{modify}, there exists a handlebody $R_{1}$ containing $R_{0}$ so that
\begin{itemize}
\item $\pi_{1}(\partial R_{1})\rightarrow \pi_{1}(\overline{R_{1}\setminus R_{0}})$ is injective;
\item $\pi_{1}(\partial R_{1})\rightarrow \pi_{1}(\overline{M\setminus R_{1}})$ is injective;
\item $R_{0}$ is contractible in $R_{1}$;
\item $\partial R_{1}$ is a union of $\partial R_{1}\cap \partial N_{j_{1}}$ and some disjoint closed discs. Therefore, $\pi_{1}(\partial R_{1}\cap \partial N_{j_{1}})\rightarrow \pi_{1}(\partial R_{1})$ is surjective. 
\end{itemize}
In particular, since $R_{0}$ is not contained in a $3$-ball in $M$, $R_{1}$ has the same property.

 Suppose that there exists a handlebody $R_{k-1}$ and a positive integer $j_{k-1}$ satisfying (1), (2), (3) and (4) in Theorem \ref{H}.

As the existence of $N_{j_{1}}$, there exists a handlebody $N_{j_{k}}$ containing $R_{k-1}$ satisfying that $R_{k-1}$ is homotopically trivial in $N_{j_{k}}$. We use Theorem \ref{modify} to find an effective handlebody $R_{k}$ relative to $R_{k-1}$ satisfying (2), (3) and (4).

Since the map $\pi_{1}(\partial R_{k-1})\rightarrow \pi_{1}(\overline{R_{k-1}\setminus R_{0}})$ is injective, $R_{k-1}$ is an effective handlebody relative to $R_{0}$ (from Lemma \ref{injective}).  Lemma \ref{upper} shows that $R_{k}$ is an effective handlebody relative to $R_{0}$. We apply Lemma \ref{injective} again and get that $R_{k}$ also satisfies (1). This finishes the proof. 
\end{proof}

\section{Minimal Surfaces and Laminations}

In Sections 5 and 6, we will talk about minimal surfaces in a contractible $3$-manifolds. In the following two sections, we will use the notions: 
\begin{itemize}[leftmargin=5pt]
\item Let $(M, g)$ be a complete $3$-manifold which is not homeomorphic to $\mathbb{R}^3$; 
\item The manifold $M$ is an increasing union of handlebodies $\{N_k\}_{k=0}^\infty$. 
\item  We may assume that $N_k$ is homotopically trivial in $N_{k+1}$ for each $k$ and $N_0$ is no contained in a $3$-ball in $M$ (see Remark \ref{ball}). 
\end{itemize}

In addition, for each $k$, the genus of $N_{k}$ is greater than zero. (If not, there is some handlebody $N_{k}$ of genus zero, namely a $3$-ball, which is in contradiction with the assumption.)

\subsection{Minimal Laminations} From Lemma \ref{sys}, each $N_{k}$ has a system of meridians $\{\gamma^{l}_{k}\}_{l=1}^{g(N_{k})}$, where $g(N_{k})$ is the genus of $N_{k}$. Our target  is to construct a lamination $\mathscr{L}_{k}:=\cup_{l}\Omega^{l}_{k}\subset N_{k}$ (i.e. a disjoint union of embedded surfaces) with $\partial \Omega^{l}_{k}=\gamma^{l}_{k}$ and ``good'' properties.

As in \cite{W1}, we use a result of Meeks and Yau (See [Theorem 6.26 Page 244] in \cite{CM1}) to construct them. However, it requires a geometric condition that  the boundary of $ N_k$ is mean convex. Then we construct a new metric $g_{k}$ over $N_{k}$ so that 

1) $g_{k}|_{N_{k-1}}$ is equal to $g|_{N_{k-1}}$; 

2) The boundary $\partial N_{k}$ is mean convex for $g_{k}$. 
 
\vspace{2mm}

As in Section 5.1 of \cite{W1}, the metric $g_{k}$ is constructed as follows:

\noindent\emph{Let $h(t)$ be a positive smooth function on $\mathbb{R}$ so that $h(t)=1$, for any $t\in \mathbb{R}\setminus [-\epsilon, \epsilon]$. Consider the function $f(x):=h(d(x, \partial N_{k}))$ and the metric $g_{k}:=f^{2}g|_{N_{k}}$. For the metric $ g_{k}$, the mean curvature $\hat{H}(x)$ of $\partial N_{k}$ is $$\hat{H}(x)=h^{-1}(0)(H(x)+2h'(0)h^{-1}(0))$$  
Choosing $\epsilon$ small enough and a function $h$ with $h(0)=2$ and $h'(0)>2\max_{x\in\partial N_{k}} |H(x)|+2$, one gets the required metric $g_{k}$.}

\vspace{2mm}
Let us describe the inductive construction of $\{\Omega^{l}_{k}\}_{l=1}^{g(N_{k})}$.  

When $l=1$, there is an embedded area-minimizing disc $\Omega^{1}_{k}\subset N_{k}$ with boundary $\gamma^1_{k}$ for the metric $g_k$ (see Theorem 6.28 of \cite{CM1}).
 
 Suppose that there are $l$ disjointly embedded stable minimal discs $\{\Omega^{i}_{k}\}_{i=1}^{l}$ with $\partial \Omega^{i}_{k}=\gamma^{i}_k$. 
 
\begin{figure}[H]
\begin{center}
\begin{tikzpicture}
\node[draw,ellipse,minimum height=150pt, minimum width = 300pt,thick] (S) at (0,0){}; 
\draw(-2, 0) arc(200:340: 2 and 1.5);
\draw (-1.8, -0.3) arc(160: 20: 1.8 and 1);

\draw[dashed][red](0,-1) arc(90:-90: 0.4 and 0.8);
\draw[red](0, -1) arc(90: 270: 0.4 and 0.8);

 \begin{scope}
    \clip (0.4, -1.8) arc(0:360: 0.4 and 0.8);
   \draw[pattern=north west lines, pattern color=red] (-4,-4) rectangle (10,10);    
     \end{scope}

\draw (-4.5,-4)--(4.5, -4);

\draw(-4.5, -5.6)--(4.5, -5.6);
\draw[dashed][red](-4.5,-4) arc(90:-90: 0.4 and 0.8);
\draw[red](-4.5, -4) arc(90: 270: 0.4 and 0.8);

 \begin{scope}
    \clip (-4.1, -4.8) arc(0:360: 0.4 and 0.8);
   \draw[pattern=north west lines, pattern color=red] (-8,-7) rectangle (10,10);    
     \end{scope} 
     
    \begin{scope}
    \clip (4.9, -4.8) arc(0:360: 0.4 and 0.8);
   \draw[pattern=north west lines, pattern color=red] (-25,-20) rectangle (100,100);    
     \end{scope}

\draw[red](4.5,-4) arc(90:-90: 0.4 and 0.8);
\draw[red](4.5, -4) arc(90: 270: 0.4 and 0.8);

\draw[->] (0, -3) -- (0, -3.5);

\node(Os) at (0,-2) {$\Omega^{1}_{k}$};

\node(Os) at (0,2) {$(N_{k}, g_{k})$};
\node(Os) at (-0.7, -2) {$\gamma^{1}_{k}$};
\node(Os) at (0, -5) {$({T}_{k,1}, g_{k}|_{T_{k, 1}})$};

\node(Os) at (-4.4,-4.8) {${\Omega^{1}_{k}}^{-}$};
\node(Os) at (-5.2, -5) {${\gamma^{1}_{k}}^{-}$};
\node(Os) at (4.6,-4.8) {${\Omega^{1}_{k}}^{+}$};
\node(Os) at (3.7, -5) {${\gamma^{1}_{k}}^{+}$};
\label{1}
\end{tikzpicture}

\caption{}

\end{center}
\end{figure}

 Let us consider the Riemannian manifold $(T_{k, l}, g_k |_{H_{k,l}})$, where $T_{k, l}:=N_{k}\setminus \amalg_{i=1}^l \Omega^l_{k}$. It is a handlebody of genus $g(N_{k})-l$. For example, see the above figure.  

The boundary of $(T_{k, l}, g_{k}|_{T_{k,l}})$ consists of two different parts. One is $\partial N_{k}\setminus \amalg_{i=1}^{l}\gamma^{i}_{l}$. The mean curvature is nonnegative on this part. The other consists of  $2l$ disjoint discs $\{{ \Omega_{k}^{i}}^{-}\}_{i=1}^{l}$ and $\{{\Omega_{k}^{i}}^{+}\}_{i=1}^l$.  The two discs ${ \Omega_{k}^{i}}^{-}$ and ${\Omega_{k}^{i}}^{+}$ are  two sides of  the same minimal disc $\Omega^i_{k}$. The mean curvature vanishes on these discs.

Therefore,  the mean curvature of the boundary of $(T_{k,l}, g_{k}|_{H_{k,l}})$ is non-negative. In addition, $\{\gamma^{i}_{k}\}_{i>l}$ is a system of meridians of the handlebody $(T_{k, l}, g_{k}|_{T_{k,l}})$ 
 
 Then, we use the result of Meeks and Yau (See Theorem 6.28 of \cite{CM1}) to find an embedded stable minimal surface $\Omega^{l+1}_{k}$ in the closure of $(T_{k, l}, g_{k} |_{T_{k,l}})$ with boundary $\gamma^{l+1}_{k}$. The disc $\Omega^{l+1}_{k}$ intersects the boundary of $(T_{k, l}, g_{k} |_{T_{k,l}})$ transversally. Hence, $\text{Int} ~\Omega^{l+1}_{k}$ is contained in $\text{Int}~ T_{k,l}$. That is to say, $\{\Omega^{i}_{k}\}_{i=1}^{l+1}$ are disjoint stable minimal surfaces for  $( N_{k}, g_{k})$.

This finishes the inductive construction.

To sum up, there exist $g(N_{k})$ disjointly embedded meridian discs $\{\Omega^{l}_{k}\}^{g_{k}}_{l=1}$. Define the lamination $\mathscr{L}_{k}$ by $\amalg_{l} \Omega^{l}_{k}$. It is a stable minimal lamination for  the new metric $g_{k}$ and for the original one away from $\partial N_{k}$ (near $N_{k-1}$, for example).

The set $\mathscr{L}_{k}\cap N_{k-1}$ is a stable minimal lamination in $(M , g)$. Each leaf has its boundary contained in $\partial N_{k-1}$.

We know that each lamination $\mathscr{L}_{k}$ intersects $N_{0}$. The reason is below:

If the set $\mathscr{L}_{k}\cap N_{0}$ is empty, we choose a tubular neighborhood $N(\mathscr{L}_{k})$ in $N_{k}$ with small radius so that the set $N(\mathscr{L}_{k})\cap N_{0}$ is also empty. That is to say, $N_{0}$ lies in the handlebody $N_{k}\setminus N(\mathscr{L}_k)$ of genus zero (i.e. a $3$-ball). This is in contradiction with our assumption that $N_{0}$ is not contained in a $3$-ball.

\subsection{Limits of laminations} First, we recall a classical convergence theorem for minimal surfaces.

\begin{definition}\label{converge} In a complete Riemannian $3$-manifold $(M, g)$, a sequence $\{\Sigma_{n}\}$ of immersed minimal surfaces \emph{converges smoothly with finite multiplicity} (at most $m$) to an immersed minimal surface $\Sigma$, if for each point $p$ of $\Sigma$, there is a disc neighborhood $D$ in $\Sigma$ of $p$, an integer $m$ and  a neighborhood $U$ of $D$ in $M$ (consisting of geodesics of $M$ orthogonal to $D$ and centered at the points of $D$) so that for $n$ large enough, each $\Sigma_{n}$ intersects $U$ in at most $m$ connected components. Each component is a graph over $D$ in the geodesic coordinates.  Moreover, each component converges to $D$ in $C^{2, \alpha}$-topology as $n$ goes to infinity. 

Note that in the case that each $\Sigma_{n}$ is embedded, the surface $\Sigma$ is also embedded. The \emph{multiplicity} at $p$ is equal to the number of connected component of $\Sigma_{n}\cap U$ for $n$ large enough. It remains constant on each component of $\Sigma$.  

\end{definition}

\begin{remark}\label{graph} Let us consider a family $\{\Sigma_{n}\}_{n}$ of properly embedded minimal surfaces  converging to the minimal surface $\Sigma$ with finite multiplicity. 
Fix a compact simply-connected subset $D\subset \Sigma$. Let $U$ be the tubular neighborhood of $D$ in $M$ with radius $\epsilon$ and  $\pi: U \rightarrow D$ be  the projection from $U $ onto $D$. It follows that the restriction $\pi|_{\Sigma_{n}\cap U}: \Sigma_{n}\cap U\rightarrow D$ is a $m$-sheeted covering map for $\epsilon$ small enough and $n$ large enough, where $m$ is the multiplicity. 

Therefore, the restriction of $\pi$ to each component of $\Sigma_{n}\cap U$ is also a covering map. Since $D$ is simply-connected, it is bijective. Therefore, each component of $\Sigma_{n}\cap U$ is a normal graph over $D$.
\end{remark}

\begin{theorem}\label{convergence}\textnormal{(see \cite{And}, Compactness Theorem on Page 96 and \cite{MRR} , Theorem 4.37 on Page 49)} Let $\{\Sigma_{k}\}_{k\in \mathbb{N}}$ be a family of properly embedded minimal surfaces  in a 3-manifold $M^{3}$ satisfying (1) each $\Sigma_{k}$ intersects a given compact set $K_{0}$; (2) for any compact set $K$ in $M$, there are three constants $C_{1}=C_{1}(K)>0$, $C_{2}=C_{2}(K)>0$ and $j_{0}=j_{0}(K)\in \mathbb{N}$ such that for each $k\geq j_{0}$, it holds that
\begin{enumerate}
\item $|A_{\Sigma_{k}}|^{2}\leq C_{1}$ on $K\cap\Sigma_{k}$, where $|A_{\Sigma_{k}}|^{2}$ is the square length of the second fundamental form of $\Sigma_{k}$
\item $\text{Area}(\Sigma_{k}\cap K)\leq C_{2}$.
\end{enumerate}
Then, after passing to a subsequence, $\Sigma_{k}$ converges to a properly embedded minimal surface with finite multiplicity in the $C^{\infty}$-topology. 
\end{theorem}
Note that the limit surface may be non-connected. \par

\vspace{2mm}

Let us consider  the sequence $\{\mathscr{L}_{k}\}$ and its limit. However, this sequence may not have Condition (2) in Theorem \ref{convergence} (see Section 5.2 of \cite{W1}). 

Therefore, $\{\mathscr{L}_{k}\}$ may not sub-converge with finite multiplicity. To overcome this, we consider the convergence to a lamination. 
Colding-Minicozzi's theory \cite{CM} shows that this sequence $\{\mathscr{L}_{k}\}$ sub-converges. Precisely, from [Proposition B.1, Page 610] in \cite{CM},  this sequence sub-converges to a minimal lamination $\mathscr{L}$ in $(M, g)$. Furthermore, we have that 

\begin{theorem}\label{min-lim}\textnormal{(\cite{W1}, Theorem 5.8 and 5.9 on Page 18)} If $(M, g)$ has positive scalar curvature, each leaf in $\mathscr{L}$ is a complete (non-compact) stable minimal surface.
\end{theorem} 

\subsection{Properness of the Limit Surfaces}  To sum up,  there is  a family $\{\mathscr{L}_{k}\}_{k}$ of laminations  sub-converging to a lamination $\mathscr{L}$. Each leaf in $\mathscr{L}$ is a complete (non-compact) embedded stable minimal surface in $(M, g)$. (See Theorem \ref{min-lim})

The remaining question is whether each leaf is properly embedded. The following theorem gives an answer. 

\begin{theorem}\label{proper}\textnormal{(\cite{W1}, Theorem 5.10 on Page 18)} Let $(M, g)$ be a complete oriented 3-manifold with positive scalar curvature $\kappa(x)$. Assume that $\Sigma$ is a complete non-compact stable minimal surface in $M$. Then, one has,$$\int_{\Sigma}\kappa(x)dv\leq 2\pi,$$where $dv$ is the volume form of the induced metric $ds^{2}$ over $\Sigma$. Moreover, if $\Sigma$ is an embedded surface, then $\Sigma$ is proper.

\end{theorem}

We will prove it in Appendix B.

\section{The Vanishing Property}
Let $(M, g)$ be a complete contractible Riemannian $3$-manifold of positive scalar curvature and  $\Sigma\subset (M, g)$ a complete (non-compact) embedded stable minimal surface. From [Theorem 2, Page 211] in \cite{SY} and Theorem \ref{proper}, the surface $\Sigma$ is a properly embedded plane (i.e.  diffeomorphic to $\mathbb{R}^{2}$). 

In the genus one case,  the geometry of such a stable minimal surface is constrained by Property P (see Theorem 4.2 of \cite{W1}). In general, its geometry is related with the fundamental group at infinity.  

\vspace{2mm}

Let $(M, g)$ and  $\{N_k\}$ be assumed as in Section 5. Theorem \ref{H} gives  an increasing family $\{R_{k}\}_{k}$ of closed handlebodies with Property $(H)$. 
 
\begin{definition}\label{Vanishing} A complete embedded stable minimal surface $\Sigma\subset (M, g)$ is called to satisfy the \emph{Vanishing Property} with respect to  $\{R_{k}\}_{k}$, if  there exists a positive integer $k(\Sigma)$ so that for any $k\geq k(\Sigma)$, any circle in $\Sigma\cap \partial R_{k}$ is contractible in $\partial R_{k}$.

Let $\mathscr{L}\subset (M, g)$ be  a stable minimal lamination  where each leaf is a complete (non-compact) stable minimal surface. It is said to have the \emph{Vanishing Property} with respect to  $\{R_{k}\}_{k}$, if there is a positive integer $k_{0}$ so that for any $k\geq k_{0}$ and each leaf $L_{t}$ in $\mathscr{L_{t}}$, then any circle in $L_{t}\cap \partial R_{k}$ is contractible in $\partial R_{k}$. 
\end{definition}

 We will prove in Corollary \ref{trivial-van1} and Theorem \ref{van2} that if $\pi^{\infty}_{1}(M)$ is trivial, any stable minimal lamination has the Vanishing property with respect to  $\{R_{k}\}_{k}$.

\begin{lemma}\label{van1} Let $(M, g)$ be a complete contractible Riemannian $3$-manifold with positive scalar curvature $\kappa(x)>0$ and $\{R_{k}\}_{k}$ a family of handlebodies  with Property (H).  If there is a complete embedded stable minimal  surface $\Sigma$ which does not satisfy the Vanishing property with respect to  $\{R_{k}\}_{k}$, then $\pi^{\infty}_{1}(M)$ is non-trivial. 
\end{lemma}
Roughly, there is a sequence of non-trivial circles in $\Sigma$ going to infinity. This sequence gives a non-trivial element in $\pi_{1}^{\infty}(M)$. 
\begin{proof} 
Since $\Sigma$ does not satisfy the Vanishing property with respect to  $\{R_{k}\}$,  there is a sequence $\{k_{n}\}_{n}$ of increasing integers  so that for each $k_{n}$, there is a circle $\gamma_{n}\subset \partial R_{k_{n}}\cap \Sigma$ which is not nullhomotopic in $\partial R_{k_{n}}$.
By [Theorem 2, Page 211] in \cite{SY}, $\Sigma$ is diffeomorphic to $\mathbb{R}^{2}$. Each $\gamma_{n}$ bounds a unique closed disc $D_{n}\subset \Sigma$.

\noindent \textbf{Remark:} The circle $\gamma_{n}$ may not be a meridian of $R_{k_{n}}$. Property (H) implies that $\pi_{1}(\partial R_{k_{n}})\rightarrow \pi_{1}(\overline{M\setminus R_{k_{n}}})$ is injective (see Definition \ref{P H}). Corollary \ref{general one} implies that    $D_{n}$ contains at least one meridian of $R_{k_{n}}$. 

\vspace{2mm}

Without loss of generality, we may assume that $\gamma_{n}$ is a meridian of $R_{k_{n}}$ and $\text{Int} ~D_{n}$ has no meridian of $R_{k_{n}}$. (If not, we can replace $\gamma_{n}$ by the meridian in $\text{Int} ~D_{n}$).

Since $\{\gamma_{n}\}_{n}$ is a collection of disjointly embedded circles in $\Sigma$, one of the following holds: for each $n$ and $n'$
\begin{center}
  (1)~$D_{n}\subset D_{n'}$; \quad 
 (2)~$D_{n'}\subset D_{n}$;\quad 
(3)~$D_{n}\cap D_{n'}=\emptyset$. 
\end{center}
Based on our assumption,  we know that

\quad \quad\quad \quad  \emph{$(\ast):$ for any $n'>n$, $D_{n}\subset D_{n'}$ or $D_{n}\cap D_{n'}=\emptyset$.}

 \noindent The reason is below: 
  If not, $D_{n'}$ is a subset of $D_{n}$. We use the  argument in the above remark to find a meridian curve in $D_{n'}\cap \partial R_{k_n} \subset \text{Int} D_n$, which is in contradiction with the above assumption.  

\vspace{2mm}

We will first show that there is an increasing subsequence of $\{D_{n}\}$ and  use  the subsequence to find  a non-trivial element in $\pi_{1}^{\infty}(M)$.

\vspace{2mm}

\noindent\emph{\textbf{Step 1}: the existence of the ascending subsequence of $\{D_{n}\}$}.

We argue by contradiction. Suppose that these is no ascending subsequence in $\{D_{n}\}$. 
Consider the partially ordered set $(\{D_{n}\}_{n}, \subset)$ induced by the inclusion. Let $C_{min}$ be the set of minimal elements in $(\{D_{n}\}_{n}, \subset)$. These discs in $C_{min}$ are disjoint in $\Sigma$. 

If the set $C_{min}$ is finite, we consider the integer $n_{0}:=\max\{n|~D_{n}\in C_{min}\}$. From the above fact $(\ast)$, the subsequence $\{D_{n}\}_{n>n_{0}}$ is an increasing subsequence, which contradicts  our hypothesis. Thus, we can conclude that the set $C_{min}$ is infinite. Namely, there is a subsequence $\{D_{n_{s}}\}_{s}$ of disjointly embedded discs.

\vspace{2mm}

 From Remark \ref{inj}, the map $\pi_{1}(\partial R_{k_{n_{s}}})\rightarrow \pi_{1}(M\setminus R_{0})$ is injective. Then, the disc $D_{n_{s}}$ intersects $R_{0}$.  
 Choose $x_{{s}}\in R_{0}\cap D_{n_{s}}$ and $r_{0}=\frac{1}{2}\min\{i_{0}, r\}$, where $r:=d^{M}(\partial R_{0},\partial R_{1})$ and  $i_{0}:=\inf_{x\in R_{1}}(\text{Inj}_{M}(x))$. Hence, the geodesic ball $ B(x_{{s}}, r_{0})\subset M$ lies in $R_{1}$.

 We apply [Lemma 1, Page 445] of \cite{YM} to the minimal surface $D_{n_{s}}\cap R_{1}$ in $(R_{1}, \partial R_{1})$ and  obtain 
\[\text{Area}(D_{n_{s}}\cap B(x_{{s}}, r_{0}))\geq C_{1}(K, i_{0}, r_{0}), \]
where $C_1$ is a constant independent of $n_s$ and $K$ is the bound of the sectional curvature on $R_1$. This leads to a contradiction from Theorem \ref{proper} as follows: 
\begin{equation*}
\begin{split}
2\pi\geq \int_{\Sigma}\kappa dv &\geq \int_{R_{1}\cap \Sigma}\kappa dv\geq \sum_{s}\int_{D_{n_{s}}\cap B(x_{{s}}, r_{0})} \kappa dv\\
&\geq \sum_{s} C\text{Area}(D_{n_{s}}\cap B(x_{{s}}, r_{0}))\\
&\geq \sum_{s}CC_{1}=\infty, 
\end{split}
\end{equation*}where  $C:=\inf_{x\in R_{1}}\kappa(x)>0$. 

\vspace{2mm}

Then, we can conclude that there is an ascending subsequence of $\{D_{n}\}_{n}$.  From now on, we abuse the notation and write $\{D_{n}\}$ for a ascending subsequence. 

\vspace{2mm}

\noindent\emph{\textbf{Step 2}: $\pi^{\infty}_{1}(M)$ is non-trivial.}

\vspace{2mm}

\noindent\textbf{Claim:} There is an integer $n_0$ so that for $n\geq n_0$, $(D_{n}\setminus D_{n-1})\cap R_{0}$ is empty. 

\vspace{2mm}

We argue by contradiction. Suppose that there exists a family $\{n_{l}\}$ of increasing integers such that $D_{n_{l}}\setminus D_{n_{l-1}}$ intersects $R_{0}$. 

Choose $x_{{l}}\in D_{n_{l}}\setminus D_{n_{l-1}}\cap R_{0}$. Hence, the geodesic ball $B(x_{l}, r_{0})\subset M$ is contained in $R_{1}$, where $r_{0}$ is assumed as above. We again apply [Lemma 1, Page 445] in \cite{YM} to the minimal surface $D_{n_{l}}\setminus D_{n_{l-1}}\cap R_{1}$ in $(R_{1}, \partial R_{1})$. 
$$\text{Area}((D_{n_{l}}\setminus D_{n_{l}-1})\cap B(x_{{l}}, r_{0}))\geq C_{1}(K, i_{0}, r_{0}).$$
From Theorem \ref{proper}, one gets a contradiction as follows:  
\begin{equation*}
\begin{split}
2\pi&\geq\int_{\Sigma}\kappa dv \geq \int_{R_{1}\cap \Sigma}\kappa dv\\&\geq \sum_{l}\int_{(D_{n_{l}}\setminus D_{n_{l-1}})\cap B(x_{{l}}, r_{0})} \kappa dv\\
&\geq \sum_{l} C\text{Area}(B(x_{{l}}, r_{0})\cap D_{n_{l}}\setminus D_{n_{l-1}})\\
&\geq C\sum_{l}C_{1}=\infty
\end{split}
\end{equation*}
It completes the proof of the claim. 
 
 \vspace{2mm}
 
 Therefore, for $n>n_0$, $\gamma_{n}$ is homotopic to $\gamma_{n_0}$ in $M\setminus R_{0}$ and not null-homotopic in $M\setminus R_{0}$.

 Because $\cup_{k}R_{k}$ may not be equal to $M$, the sequence $\{\gamma_{n}\}_{n>n_0}$ of circles may not go to infinity. In order to overcome it,  we choose a new family $\{\gamma'_{n}\}_{n>n_0}$ of circles going to infinity to replace it. 
 
The map $\pi_{1}(\partial R_{k_{n}}\cap \partial N_{j_{k_{n}}})\rightarrow \pi_{1}(\partial R_{k_{n}})$ is surjective (see Theorem \ref{H} and Definition \ref{P H}).  Hence, we can find a circle $\gamma'_{n}\subset \partial N_{j_{k_{n}}}\cap \partial R_{k_{n}}$ which is homotopic to $\gamma_{n}$ in $\partial R_{k_{n}}$.  The sequence of circles $\{\gamma'_{n}\}_{n\geq n_0}$ goes to infinity.

 The sequence $\{\gamma'_{n}\}$ also has the property that for $n>n_0$, \begin{itemize} \item $\gamma'_{n}$ is homotopic to $\gamma'_{n+1}$ in $M\setminus R_{0}$; \item$\gamma'_{n}$ is not null-homotopic in $M\setminus R_{0}$. \end{itemize} 
 From Remark \ref{element}, $\pi_{1}^{\infty}(M)$ is not trivial.  
\end{proof}

As a corollary, we have 
\begin{corollary}\label{trivial-van1} Let $(M, g)$ be a complete Riemannian $3$-manifold of positive scalar curvature and $\{R_{k}\}_{k}$ a family of handlebodies with Property (H). If $\pi_{1}^{\infty}(M)$ is trivial, then any complete stable minimal surface in $(M, g)$ has the Vanishing property with respect to $\{R_{k}\}_{k}$. 
\end{corollary}

\begin{theorem}\label{van2} Let $(M, g)$ be a complete Riemannian $3$-manifold of positive scalar curvature and  a family of handlebodies $\{R_{k}\}_{k}$ with Property (H). If each leaf in a lamination $\mathscr{L}$ is a complete (non-compact) stable minimal surface satisfying the Vanishing property with respect to  $\{R_{k}\}_{k}$, then the lamination $\mathscr{L}$ also has the Vanishing property with respect to  $\{R_{k}\}_{k}$.
\end{theorem}

\begin{proof} We argue by contradiction. Suppose that there exists a sequence $\{L_{t_{n}}\}$ of leaves in $\mathscr{L}$ and a sequence $\{k_{n}\}_{n}$ of increasing integers  so that  some circle $\gamma_{n} \subset L_{t_{n}}\cap \partial R_{k_{n}}$ is not homotopically trivial  in $\partial R_{k_{n}}$ for each $n$. 

 The leaf $L_{t_{n}}$ is a complete (non-compact) stable minimal surface.  From [Theorem 2, Page 211] in \cite{SY}, it is diffeomorphic to $\mathbb{R}^{2}$. The circle $\gamma_{n}$ bounds a unique closed disc $D_{n}\subset L_{t_{n}}$.  Since $\gamma_{n}$ is not null-homotopic in $\overline{M\setminus R_{0}}$ (See Remark \ref{inj}), the disc $D_{n}$ intersects $R_{0}$.
 
 \vspace{2mm}
\noindent \textbf{Step 1}: \emph{The sequence $\{L_{t_{n}}\}_{n}$ sub-converges smoothly with finite multiplicity.}

\vspace{2mm}
Since each $L_{t_{n}}$ is a stable minimal surface, we use [Theorem 3, Page 122] in \cite{Schoen} to show that, for a fixed compact set $K\subset M$, there exists a constant $C_{1}=C_{1}(K, M, g)$ satisfying that 
$$|A_{L_{t_{n}}}|^{2}\leq C_{1}~\text{on}~K\cap L_{t_{n}}$$
where $|A_{L_{t_{n}}}|^{2}$ is the squared norm of the second fundamental form of $L_{t_{n}}$. 

From Theorem \ref{proper}, $\int_{L_{t_{n}}}\kappa dv\leq 2\pi$, hence
$$\text{Area}(K\cap L_{t_{n}})\leq 2\pi(\inf_{x\in K} \kappa(x))^{-1}.$$ 
From Theorem \ref{convergence}, the sequence $\{L_{t_{n}}\}_{n}$ smoothly sub-converges  to a sublamination $\mathscr{L}'$  of $\mathscr{L}$ with finite multiplicity. In addition, $\mathscr{L}'$  is also properly embedded (see Theorem \ref{convergence}). 

The lamination $\mathscr{L}'$ may has infinitely many components. Let $\mathscr{L}''\subset \mathscr{L}'$ be a set of leaves  intersecting $R_{0}$. 
Since $\mathscr{L}'$ is properly embedded, $\mathscr{L}''$ has finitely many leaves. 

\vspace{1mm}

Since each leaf $L_{t}$ in $\mathscr{L'}$ is homeomorphic to $\mathbb{R}^2$, an embedded circle $\gamma\subset \partial R_{k}\cap L_{t}$ bounds a unique disc $D\subset L_{t}$ for $k>0$. 

If $L_{t}$ is in $\mathscr{L}'\setminus \mathscr{L}''$, the intersection $D\cap R_{0}$ is empty. Namely, $\gamma$ is homotopically trivial in $\overline{M\setminus R_{0}}$. Since the map $\pi_{1}(\partial R_{k})\rightarrow \pi_{1}(\overline{M\setminus R_{0}})$ is injective for $k>0$ (see Remark \ref{inj}), then $\gamma$ is null-homotopic in $\p R_{k}$. 

Therefore, we can conclude that for any $k>0$ and any leaf $L_{t}\in \mathscr{L}'\setminus \mathscr{L}''$, any circle in $L_{t}\cap \p R_{k}$ is homotopically trivial in $\partial R_{k}$. 

\vspace{2mm}

\noindent\textbf{Step 2:} \emph{The Vanishing property gives a contradiction.}

\vspace{2mm}

From now on, we abuse the notation and write $\{L_{t_{n}}\}$ for a convergent sequence. In addition, we assume the lamination $\mathscr{L}'':=\amalg^{m}_{s=1} L_{t_{s}}$ ($\mathscr{L}''$ has finitely many leaves).

The Vanishing property gives an integer $k(L_{t_{s}})$ for $L_{t_{s}}$. For $k\geq \sum_{s=1}^{m} k(L_{t_{s}})$, any circle in $\partial R_{k}\cap \mathscr{L}''$ is contractible in $\partial R_{k}$.  From the above fact, for $k>0$, any closed curve in $\partial R_{k}\cap \mathscr{L}'\setminus \mathscr{L}''$ is also homotopically trivial in $\partial R_{k}$. 

Therefore, for any $k\geq\sum_{s=1}^{m} k(L_{t_{s}})$,   any circle in $\partial R_{k}\cap \mathscr{L}'$ is homotopically trivial  in $\partial R_{k}$.  

\vspace{2mm}

We fix the integer $k\geq \sum^{m}_{s=1} k(L_{t_s})$ and have the following: 

\vspace{2mm}

\noindent\textbf{Claim:} For $n$ large enough, any circle in $\p R_{k}\cap L_{t_{n}}$ is homotopically trivial in $\p R_{k}$. 

\vspace{2mm}

 We may assume that $\mathscr{L}'$ intersects $\partial R_{k}$ transversally. Since $\mathscr{L}'$ is properly embedded, $\partial R_{k}\cap \mathscr{L}'$ has finitely many components.  Each component of $\partial R_{k}\cap \mathscr{L}'$ is an embedded circle. From the above fact, it is homotopically trivial in $\partial R_{k}$. That is to say, 
 \begin{equation*}
\pi_{1}(\partial R_{k}\cap \mathscr{L}')\rightarrow \pi_{1}(\p R_{k})~ \text{is a trivial map}. 
 \end{equation*}

Choose an open tubular neighborhood $U$ of $\mathscr{L}'\cap \partial R_{k}$ in $\partial R_{k}$. It is  homotopic to $\p R_{k}\cap \mathscr{L'}$ in $\p R_{k}$. Thus, the map $\pi_{1}(U)\rightarrow \pi_{1}(\partial R_{k})$ is also a trivial map.

Since  $\{L_{t_{n}}\}$ converges to $\mathscr{L}'$,  $L_{t_{n}}\cap \partial R_{k}$ is contained in $U$ for $n$ large enough. Hence, the map $\pi_{1}(\p R_k\cap L_{t_{n}})\rightarrow \pi_1(\p N_{k})$ is trivial. Namely, any circle in $\partial R_{k}\cap L_{t_{n}}$ is homotopically trivial in $\partial R_{k}$. The claim follows. 

\vspace{2mm}

  The boundary $\gamma_{n}\subset \partial R_{k_{n}}$ of $D_{n}$ is non-contractible in $\partial R_{k_{n}}$.  It is also non-contractible in $\overline{M\setminus R_{0}}$ (see Remark \ref{inj}). If $k_{n}> k$, we use  Corollary \ref{general one} to find a meridian $\gamma' \subset L_{t_{n}}\cap \p R_{k}$ of $R_{k}$. This is in contradiction with the above claim. \end{proof}

As a consequence, we have 
\begin{corollary}\label{trivial-van2} Let $(M, g)$ be a complete contractible Riemannian manifold of positive scalar curvature and $\{R_{k}\}_{k}$ a family of handlebodies with Property (H). If $\pi_{1}^{\infty}(M)$ is trivial, then any complete stable minimal lamination in $(M, g)$ has the Vanishing property with respect to $\{R_{k}\}_{k}$. 
\end{corollary}

\section{Proof of Main Theorems}

In  Sections 7 and 8, we will complete the proof of Theorem \ref{A}. We will use the following notions and assumptions: 

\begin{itemize}[leftmargin=15pt]
\item Let $(M, g)$ be a complete contractible $3$-manifold with trivial $\pi^\infty_1(M)$ and with positive scalar curvature. 
\item The manifold $M$ is an increasing union of closed handlebodies $\{N_k\}_k$. 
\item Assume that $M$ is not homeomorphic to $\mathbb{R}^3$. 
\item As in Remark \ref{ball}, we may assume that each $N_{k}$ is homotopically trivial in $N_{k+1}$ and none of the $N_{k}$ is contained in a $3$-ball (see Remark \ref{ball}). In addition, the genus of $N_{k}$ is greater than zero, for $k>0$. 
\item Each $N_k$ has a system of meridians $\{\gamma^{l}_{k}\}_{l=1}^{g(N_{k})}$. There is  a lamination $\mathscr{L}_{k}:=\amalg_{l} \Omega^{l}_{k}\subset N_k$, where each leaf $\Omega^{l}_{k}$ is a disc with boundary $\gamma^{l}_{k}\subset \partial N_{k}$
\item As in Section 5, $\mathscr{L}_k$ sub-converges to a  lamination $\mathscr{L}:=\cup_{t\in \Lambda}L_{t}$, where each leaf $L_{t}$ is a complete (non-compact) stable minimal surface. 

\end{itemize}

Our main approach is to argue by contradiction. We suppose that $(M, g)$ has positive scalar curvature. 
 Each leaf in $\mathscr{L}$ is a properly embedded plane (see Theorem \ref{proper} and a result of Schoen and Yau \cite{SY}).

 \vspace{2mm}
 
We now study the lamination $\mathscr{L}$ and its relationship with  Vanishing property.

\emph{Since $\pi^{\infty}_{1}(M)$ is trivial, there is an ascending family $\{R_{k}\}_{k}$ of  handlebodies  satisfying Property (H), so that} 
\begin{itemize}[leftmargin=15pt]
\item[(a)]\emph{the lamination $\mathscr{L}$ has the Vanishing property with respect to $\{R_{k}\}_{k}$;}
\item[(b)]\emph{ for each $k$ and any $N_{j}$ containing $R_{k}$, the intersection $\mathscr{L}_{j}\cap \partial R_{k}$ has at least one meridian of $R_{k}$.}
\end{itemize}
 
 The reason is  as follows:  since $M$ is not homeomorphic to $\mathbb{R}^{3}$, then we use Theorem \ref{H} to find  $\{R_{k}\}_{k}$.  Since $\pi^{\infty}_{1}(M)$ is trivial, Corollary \ref{trivial-van2} shows that the lamination $\mathscr{L}$ has the Vanishing Property with respect to this family. 
 
None of the $R_{k}$ is contained in a $3$-ball (see Definition \ref{P H}).  Together with Property (H), we use Corollary \ref{inter} to show that if $N_{j}$ contains $R_{k}$, the intersection $\mathscr{L}_{j}\cap \partial R_{k}$ has at least one meridian of $R_{k}$.

 \begin{remark}\label{general} In Section 8, our proof requires that $\p R_{k}$ intersects some leaf $L_{t}$ transversally. In order to overcome it, we  will deform the handlebody $R_{k}$ in a small tubular neighborhood of $\p R_{k}$ so that the boundary of the new handlebody intersects $L_{t}$ transversally. 
 
 This new handlebody also satisfies  (a) and (b). The reason is as follows: 
 
For any handlebody ${R}'_{k}$ obtained by deforming $R_{k}$,  the maps $\pi_{1}(\partial {R}'_{k})\rightarrow \pi_{1}(\overline{{R}'_{k}\setminus R_{0}})$ and $\pi_{1}(\partial {R}'_{k})\rightarrow \pi_{1}(\overline{M\setminus {R}'_{k}})$ are both injective. The proofs of the items (a) and (b) just depend on the  injectivity  of these two maps. 
Hence,  the items (a) and (b) are also true for the handlebody $R'_k$. 
 \end{remark}

We use the item (a) to show that  \par
\vspace{1mm}
\noindent\emph{There is an integer $k_{0}>0$ so that for any $k\geq k_{0}$ and any leaf $L_{t}$ in $\mathscr{L}$, any circle in ${L}_{t}\cap \partial R_{k}$ is null-homotopic  in $\partial R_{k}$.}\par
\vspace{1mm}

This fact implies a covering lemma. We will prove it in Section 8. 
\begin{lemma}\label{covering} For any $k\geq k_{0}$, $\partial R_{k}(\epsilon)\cap \mathscr{L}$ is contained in a disjoint union of finitely many closed discs in $\partial R_{k}(\epsilon)$, where $R_{k}(\epsilon):=R_{k}\setminus N_{\epsilon}(\partial R_{k})$, $N_{\epsilon}(\partial R_{k})$ is some tubular neighborhood of $\partial R_{k}$ in $R_{k}$.
\end{lemma}

We now use the covering lemma to finish the proof of Theorem \ref{A}. 

\begin{proof} Suppose that a complete contractible $3$-manifold $(M, g)$, with positive scalar curvature and trivial $\pi^{\infty}_{1}$, is not homeomorphic to $\mathbb{R}^{3}$. As above,  there is an ascending family $\{R_{k}\}_{k}$ of  handlebodies with Property (H), so that

(a) the lamination $\mathscr{L}$ has the Vanishing property with respect to  $\{R_{k}\}_{k}$; 
 
(b) for each $k$ and any $N_{j}$ containing $R_{k}(\epsilon)$, the intersection $\mathscr{L}_{j}\cap \partial R_{k}(\epsilon)$ has at least one meridian of $R_{k}(\epsilon)$.

The Vanishing property implies  Lemma \ref{covering} (We will show it in Section 8). That is to say, the intersection $\mathscr{L}\cap \partial R_{k}(\epsilon)$ is in the union of  disjoint closed discs $\{ D_{i}\}_{i=1}^{s}$ for $k\geq k_{0}$.

\vspace{2mm}

 Choose an open neighborhood $U$ of the closed set $\mathscr{L}\cap {R}_{k+1}$ so that $U\cap \partial R_{k}(\epsilon)$ is contained in a disjoint union $\amalg_{i=1}^{s}D'_{i}$, where $D'_{i}$ is an open tubular neighborhood of $D_{i}$ in $\partial R_{k}(\epsilon)$. Each $D'_{i}$ is an open disc in $\partial R_{k}(\epsilon)$. 

Since $\mathscr{L}_{k}$ subconverges to $\mathscr{L}$, there exists an integer $j$, large enough, satisfying 
 \begin{center}
 
(1)$\mathscr{L}_{j}\cap R_{k+1}\subset U$; \quad (2)
 $R_{k}(\epsilon)$ is contained in $N_{j}$. 
 \end{center}
Therefore, $\mathscr{L}_{j}\cap \partial R_{k}(\epsilon)$ is contained in $U\cap \partial R_{k}(\epsilon)\subset \amalg D'_{i}$. The induced map $\pi_{1}(\mathscr{L}_{j}\cap \partial R_{k}(\epsilon))\rightarrow \pi_{1}(\amalg_{i}D'_{i})\rightarrow \pi_{1}(\partial R_{k}(\epsilon))$ is a trivial map. We can conclude that  any circle in $\mathscr{L}_{j}\cap \partial R_{k}(\epsilon)$ is homotopically trivial  in $\partial R_{k}(\epsilon)$.

\vspace{2mm}

However, from (b), there exists a meridian $\gamma \subset \mathscr{L}_{j}\cap \partial R_{k}(\epsilon)$ of $R_{k}(\epsilon)$. This contradicts  the last paragraph and  finishes the proof of Theorem \ref{A}.
\end{proof}

\section{The proof of Lemma \ref{covering}}

 This section is the same  as Section 7 of \cite{W1}. In order to prove Lemma \ref{covering}, we introduce a set $S$ and prove its finiteness which will imply Lemma \ref{covering}. We begin with two topological lemmas. 
 
 \begin{lemma}\label{existence} Let $(\Omega, \partial \Omega) \subset (N, \partial N)$ be a 2-sided embedded disc with some closed sub-discs removed, where $N$ is a closed handlebody of genus $g>0$. Assume that each circle $\gamma_{i}$ is contractible in $\partial N$, where $\partial \Omega=\amalg_{i}\gamma_{i}$. Then $N\setminus \Omega$ has two connected components. Moreover, there is a unique component $B$ satisfying that the induced map $\pi_{1}(B)\rightarrow\pi_{1}(N)$ is trivial.  
\end{lemma} 

We will show the lemma in Appendix A. 

 \begin{lemma}\label{sep} Let $(\Omega_{1}, \partial\Omega_{1})$ and $(\Omega_{2}, \partial\Omega_{2})$ be two disjoint surfaces as assumed in Lemma \ref{existence}. Assume that for each $t=1,2$, $N\setminus \Omega_{t}$ has a unique component $B_{t}$ with the property that the map $\pi_{1}(B_{t})\rightarrow \pi_{1}(N)$ is trivial. Then  one of the following holds:
\begin{enumerate}
\item $B_{1}\cap B_{2}=\emptyset$;
\item $B_{1}\subset B_{2}$;
\item$B_{2}\subset B_{1}$.
\end{enumerate} 
 \end{lemma}
The proof is the same as the proof of  Lemma 7.2 in \cite{W1}.

\subsection{Definition of the set $S$} 

Let $(M, g)$, $\{N_k\}_k$ and $\mathscr{L}:=\amalg_{t\in \Lambda} L_t$ be assumed as in the beginning of Section 7. As in Section 7, there is a family of handlebodies $\{R_k\}_k$ satisfying that 

\vspace{2mm}

\noindent \emph{there is a positive integer $k_{0}$ so that for each $k\geq k_{0}$ and each $t\in \Lambda$, each circle in $L_{t}\cap \partial R_{k}$ is homotopically trivial in $\partial R_{k}$}. 

\vspace{2mm}

In the following, we will work on the open handlebody $\text{Int}~R_{k}$ and construct the set $S$, for a fixed integer $k\geq k_{0}$.

\subsubsection{The element in $S$} Let $\{\Sigma^{t}_{i}\}_{i\in I_{t}}$ be the set of  components of $L_{t}\cap \text{Int} ~R_{k}$ for each $t\in\Lambda$. (It may be empty.)  We will show that for each component $\Sigma^{t}_{i}$, $R_{k}\setminus \overline{\Sigma^{t}_{i}}$ has a unique component $B^{t}_{i}$ satisfying that $\pi_{1}(B^{t}_{i})\rightarrow \pi_{1}(R_{k})$ is trivial. 

\vspace{2mm}

If $L_{t}$ intersects $\p R_{k}$ transversally,  the boundary $\p \Sigma^{t}_{i}\subset L_{t}\cap \p R_{k}$ is the union of some disjointly embedded circles. From the Vanishing property, any circles in the boundary $\partial \Sigma^{t}_{i}\subset L_{t}\cap \p R_{k}$ is homotopically trivial in $\p R_{k}$. 

In  addition, since $L_{t}$ is homeomorphic to $\mathbb{R}^{2}$, $\Sigma^{t}_{i}$ is homeomorphic to an open disc with finite punctures. By Lemma \ref{existence}, $R_{k}\setminus \overline{\Sigma^{t}_{i}}$ has a unique component $B^{t}_{i}$ satisfying that $\pi_{1}(B^{t}_{i})\rightarrow \pi_{1}(R_{k})$ is trivial.

\vspace{2mm}

In general, $L_{t}$ may not intersect $\p R_{k}$ transversally. To overcome it, we will deform the surface $\p R_{k}$.  Precisely, for the leaf $L_{t}$, there is a new handlebody $\tilde{R}_{k}[\epsilon_{t}]$ containing $R_{k}$ so that $L_{t}$ intersects $\p \tilde{R}[\epsilon_{t}]$ transversally, where $\tilde{R}_{k}[\epsilon_{t}]$ is a closed tubular neighborhood of $ R_{k}$ in $M$.   

We consider the component $\tilde{\Sigma}^{t}_{i}$ of $L_{t}\cap \text{Int} \tilde{R}_{k}[\epsilon_{t}]$ containing $\Sigma^{t}_{i}$. As above,  $\tilde{R}_{k}[\epsilon_{t}]\setminus \overline{\tilde{\Sigma}^{t}_{i}}$ has a unique component $\tilde{B}^{t}_{i}$ so that the map $\pi_{1}(\tilde{B}^{t}_{i})\rightarrow \pi_{1}(\tilde{R}_{k}[\epsilon_{t}])$ is trivial. 

Choose the component $B^{t}_{i}$ of $\tilde{B}^{t}_{i}\cap R_{k}$ whose boundary contains $\Sigma^{t}_{i}$. It is a component of $R_{k}\setminus \overline{\Sigma^{t}_{i}}$.  In addition, the map $\pi_{1}(B^{t}_{i})\rightarrow \pi_{1}(\tilde{B}^{t}_{i})\rightarrow \pi_{1}(\tilde{R}_{k}[\epsilon_{t}])$ is trivial. Since $R_{k}$ and $\tilde{R}_{k}[\epsilon_{t}]$ are homotopy equivalent, the map $\pi_{1}(B^{t}_{i})\rightarrow \pi_{1}({R}_{k})$ is also trivial.  This finishes the construction of $B^{t}_{i}$. 

\vspace{2mm}

\subsubsection{The properties of $S$} From Lemma \ref{sep}, for any $B^{t}_{i}$ and $B^{t'}_{i'}$, it holds one of the following 
 \begin{center}
(1) $B^{t}_{i}\cap B^{t'}_{i'}=\emptyset$; \quad 
(2) $B^{t}_{i}\subset B^{t'}_{i'}$; \quad 
 (3) $B^{t'}_{i'}\subset B^{t}_{i}$,

\end{center}where $t, t'\in \Lambda$, $i\in I_{t}$ and $i'\in I_{t'}$. Therefore, $(\{B^{t}_{i}\}_{t\in \Lambda, i\in I_{t}}, \subset)$ is a partially ordered set. Any set can be partially ordered by inclusion regardless of whether the items (1), (2) and (3) hold. 
  
  We now  consider  the set $\{B_{j}\}_{j\in J}$ of maximal elements. However this set  may be infinite.

\begin{definition}
  $S:=\{B_{j}| B_{j}\cap R_{k}(\epsilon/2)\neq \emptyset$, for any  $j\in J\},$ where $R_{k}(\epsilon/2)$ is $R_{k}\setminus N_{\epsilon/2}(\partial R_{k})$ and $N_{\epsilon/2}(\partial R_{k})$ is a 2-sided tubular neighborhood of $\partial R_{k}$ with radius $\epsilon/2$.
\end{definition}

\begin{proposition}\label{intersect} Let $\Sigma^{t}_{i}$ be one component of $L_{t}\cap \text{Int}~R_{k}$ and $B^{t}_{i}$ assumed as above. If $B^{t}_{i}$ is an element in $S$, then $\Sigma^{t}_{i}\cap R_{k}(\epsilon/2)$ is nonempty. \end{proposition}

The proof is the same as Proposition 6.5 of \cite{W1}.

 \begin{proposition}\label{cover}
 $R_{k}(\epsilon)\cap \mathscr{L}\subset \bigcup_{{B}_{j} \in S} \overline{B}_{j}\cap R_{k}(\epsilon)$. Moreover, $\partial R_{k}(\epsilon)\cap \mathscr{L}\subset \bigcup_{{B}_{j}\in S} \overline{B}_{j}\cap \partial R_{k}(\epsilon)$.
 \end{proposition}

 The proof is the same as Proposition 6.6 of \cite{W1}.

\subsection{The finiteness of the set $S$} The set $\partial B_{j}\cap \text{Int}~R_{k}$ equals some $\Sigma^{t}_{i}\subset L_{t}$ for $t \in \Lambda$.
Let us consider the set $S_{t}:=\{B_{j}\in S| \partial B_{j}\cap \text{Int} ~R_{k}\subset L_{t} \}$. Then, $S=\amalg_{t\in \Lambda} S_{t}$. Note that each $B_{j}\in S_{t}$ is a $B^{t}_{i}$ for some $i\in I_{t}$.\par
In this subsection, we first show that each $S_{t}$ is finite. Then, we argue that $\{S_{t}\}_{t\in\Lambda}$ contains at most finitely many nonempty sets.
These imply the finiteness of $S$. 
 
\begin{lemma}\label{finite1} Each $S_{t}$ is finite.\end{lemma}
\begin{proof}

We argue by contradiction. Suppose that $S_{t}$ is infinite for some $t$. 

For each $B_{j}\in S_{t}$, there exists  a $i\in I_{t}$ so that $B_{j}$ is equal to $B^{t}_{i}$, where $B^{t}_{i}$ is a component of $R_{k}\setminus \overline{\Sigma^{t}_{i}}$ and $\Sigma^{t}_{i}$ is one component of $L_{t}\cap \text{Int}~R_{k}$. By Proposition \ref{intersect}, $\Sigma^{t}_{i}\cap R_{k}(\epsilon/2)$ is nonempty. 
 
Choose $x_{j}\in\Sigma^{t}_{i}\cap R_{k}(\epsilon/2)$ and $r_{0}=\frac{1}{2}\min\{\epsilon/2, i_{0}\}$, where $i_{0}:=\inf_{x\in R_{k}} \text{Inj}_{M}(x)$.
Then the geodesic ball $B(x_{j}, r_{0})$ in $M$ is contained in $R_{k}$.

We apply [Lemma 1, Page 445] in \cite{YM} to the minimal surface $(\Sigma^{t}_{i},\partial \Sigma^{t}_{i})\subset (R_{k}, \partial R_{k})$.  One knows that, 
$$\text{Area}( \Sigma^{t}_{i}\cap B(x_{j}, r_{0}))\geq C(r_{0}, i_{0}, K)$$
where $K=\sup_{x\in R_{k}}|K_{M}|$.  This leads to a contradiction from Theorem \ref{proper} as below:
\begin{equation*}
\begin{split}
2\pi\geq \int_{L_{t}} \kappa(x)dv
       & \geq\sum_{B_{j}\in S_{t}} \int_{\Sigma^{t}_{i}} \kappa(x) dv\geq \sum_{B_{j}\in S_{t}} \int_{\Sigma^{t}_{i}\cap B(x_{j}, r_{0})} \kappa(x) dv\\
      &\geq \inf_{x\in R_{k}}(\kappa(x))\sum_{B_{j}\in S_{t}} \text{Area}(B(x_{j}, r_{0})\cap \Sigma^{t}_{i})\\
      &\geq C\inf_{x\in R_{k}}(\kappa(x)) |S_{t}|=\infty
\end{split}
\end{equation*}
This finishes the proof.
\end{proof}

\begin{lemma}\label{finite2} $\{S_{t}\}_{t\in \Lambda}$ contains at most finitely many nonempty sets.
\end{lemma}

\begin{proof}We argue by contradiction. Suppose that there exists a sequence $\{S_{t_{n}}\}_{n\in\mathbb{N}}$ of nonempty sets.
For an element $B_{j_{t_{n}}}\in S_{t_{n}}$, there is some $i_{n}\in I_{t_{n}}$ so that $B_{j_{t_n}}$ equals $B^{t_{n}}_{i_{n}}$ where $B^{t_{n}}_{i_{n}}$ is one component of $R_{k}\setminus \overline{\Sigma^{t_{n}}_{i_{n}}}$ and $\Sigma^{t_{n}}_{i_{n}}$ is one of components of $L_{t_{n}}\cap \text{Int}~R_{k}$. Note that $\pi_{1}(B^{t_{n}}_{i_{n}})\rightarrow \pi_{1}(R_{k})$ is trivial. 

By Proposition \ref{intersect}, $\Sigma^{t_{n}}_{i_{n}}\cap R_{k}(\epsilon/2)$ is not empty. Pick a point $p_{t_{n}}$ in $\Sigma^{t_{n}}_{i_{n}}\cap R_{k}(\epsilon/2)$.

\vspace{2mm}
\noindent\emph{\textbf{Step1}: $\{L_{t_{n}}\}$ subconverges to a lamination $\mathscr{L}'\subset \mathscr{L}$ with finite multiplicity.}\par
\vspace{2mm}

Since $L_{t_{n}}$ is a stable minimal surface, by [Theorem 3, Page122] in \cite{Schoen}, for any compact set $K\subset M$, there is a constant $C_{1}:=C_{1}(K, M, g)$ such that   $$|A_{L_{t_{n}}}|^{2}\leq C_{1}~ \text{on}~ K\cap L_{t_{n}}.$$ 
 From Theorem \ref{proper}, $\int_{L_{t_{n}}}\kappa(x)dv\leq 2\pi$. Hence,  $$\text{Area}(K\cap L_{t_{n}})\leq 2\pi(\inf\limits_{x\in K}\kappa(x))^{-1}.$$

 We use Theorem \ref{convergence} to find a sub-sequence of $\{L_{t_{n}}\}$ sub-converging to a properly embedded lamination $\mathscr{L}'$ with finite multiplicity.  Since $\mathscr{L}$ is a closed set in $M$, $\mathscr{L}'\subset \mathscr{L}$ is a sub-lamination.

From now on, we abuse notation and write $\{L_{t_{n}}\}$ and $\{p_{t_{n}}\}$ for  the convergent subsequence. 

\vspace{2mm}
\noindent\emph{\textbf{Step 2}: $\{\Sigma^{t_{n}}_{i_{n}}\}$ converges  with multiplicity one.}\par
\vspace{2mm}
\begin{figure}[H]
\begin{center}
\begin{tikzpicture} 
\draw (-9,-4) -- (3, -4);
\draw (-7,4.1) -- (5,4.1);
\draw (-9, -4)--(-7, 4.1);
\draw (3, -4)--(5, 4.1);

\fill[pink]  (-0.9, 0.75) arc(0:360: 1.6 and 1.2);

\draw (2.1, 1) arc (360: 0: 4 and 3);
\draw[dashed] (2,0) arc (0: 169: 4.05 and 3);
\draw (2, 0) arc (365 :160: 4.1 and 3);

\draw (-0.9, 0.75) arc(0:360: 1.6 and 1.2);
\draw[dashed] (-1, -.025) arc(0:360: 1.5 and 1.2);
\node(Os) at (2,-3) {$L_{t_{\infty}}$};
\node(Os) at (-1, -2.5) {$\Sigma_{\infty}$};
\node(Os) at (0, 0) {$\Sigma^{t_{n}}_{i_{n}}$};
\node(Os) at (-2.4, -0.75) {$B^{\Sigma_{\infty}}(p_{\infty})$};
\node(Os) at (-6, -3) {$\pi^{-1}(B^{\Sigma_\infty}(p_{\infty}))\cap \Sigma^{t_{n}}_{i_{n}}$};

\draw[->] (-5.8, -2.7) -- (-2.4, 0.5);
 \end{tikzpicture}
 
 \caption{}
\end{center}
\end{figure}

Let $L_{t_{\infty}}$ be the unique component of $\mathscr{L}'$ passing through $p_{\infty}$, where $p_{\infty}=\lim_{n\rightarrow \infty}p_{t_{n}}$. The limit of $\{\Sigma^{t_{n}}_{i_{n}}\}$ is the component $\Sigma_{\infty}$ of $L_{t_{\infty}}\cap R_{k}$ passing through $p_{\infty}$, where $\Sigma^{t_{n}}_{i_{n}}$ is the unique component of $R_{k}\cap L_{t_{n}}$ passing though $p_{t_{n}}$.\par

Let $D\subset L_{t_{\infty}}$ be a simply-connected subset satisfying ${\Sigma}_{\infty}\subset D$. (Its existence is ensured by the fact that $L_{\infty}$ is homeomorphic to $\mathbb{R}^{2}$.) Since $\{L_{t_{n}}\}$ smoothly converges to $L_{t_{\infty}}$ with finite multiplicity, there exists  $\epsilon_{1}>0$ and an integer $n_0$ such that $$\Sigma^{t_{n}}_{i_{n}}\subset D(\epsilon_{1}),~\text{for}~n>n_0, 
 $$
where $D(\epsilon_{1})$ is the tubular neighborhood of $D$ with radius $\epsilon_{1}$ in $M$ (see Definition \ref{converge}.)\par

Let $\pi:D(\epsilon_{1})\rightarrow D$ be the projection. From Remark \ref{graph}, we know that
for $n$ large enough, the restriction of $\pi$ to each component of $L_{t_{n}}\cap D(\epsilon_{1})$ is injective. 

Hence, $\pi|_{\Sigma^{t_{n}}_{i_{n}}}:\Sigma^{t_{n}}_{i_{n}}\rightarrow D$ is injective. That is to say, $\Sigma^{t_{n}}_{i_{n}}$ is a normal graph over a subset of $D$.  
 Therefore, $\{\Sigma^{t_{n}}_{i_{n}}\}$ converges to ${\Sigma}_{\infty}$ with multiplicity one (see Definition \ref{converge}). That is to say, there is a geodesic disc $B^{\Sigma_{\infty}}(p_{\infty})\subset \Sigma_{\infty}$ centered at $p_{\infty}$ with small raduis so that 

\vspace{2mm}
\emph{$(\ast\ast)$: the set $\pi^{-1}(B^{\Sigma_{\infty}}(p_{\infty}))\cap \Sigma^{t_{n}}_{i_{n}}$ is connected and a normal graph over $B^{\Sigma_{\infty}}(p_{\infty})$, for large $n$. }
\vspace{2mm}

\vspace{2mm}
\noindent\emph{\textbf{Step 3}: Get a contradiction.}\par
\vspace{2mm}
 There exists a neighborhood $U$ of $p_{\infty}$ and a coordinate map $\Phi$, such that each component of $\Phi(\mathscr{L}\cap U)$ is $\mathbb{R}^{2}\times \{x\}\cap \Phi(U)$ for some $x\in \mathbb{R}$ (see the definition of the Lamination in [Appendix B, 609-612] of \cite{CM}). Choose the disc $B^{\Sigma_{\infty}}(p_{\infty})$ and $\epsilon_{1}$ small enough such that $\pi^{-1}(B^{\Sigma_{\infty}}(p_{\infty}))\subset U$. We may assume that $U=\pi^{-1}(B^{\Sigma_{\infty}}(p_{\infty}))$.
 
 \begin{figure}[H]
\begin{center}
\begin{tikzpicture} 

\node[draw,ellipse,minimum height=100pt, minimum width = 100pt,thick] (S) at (-4,0){}; 
\node[draw,ellipse,minimum height=100pt, minimum width = 100pt,thick] (S) at (3,0){}; 

\begin{scope}
    \clip (3,0) circle (1.75cm);
    \fill[pink] (1, 1) rectangle (7, -3);
  \end{scope}
  \begin{scope}
  \clip (-4,0) circle (1.75cm);
  \fill[pink] (-8,0.5) rectangle (1, 10);
  \end{scope}
\draw  (1,0) -- (5,0);
\draw (1,0.5) -- (5,0.5);
\draw(1, 1) -- (5,1);

\draw (-6,0) --(-2,0);
\draw (-6, 1)--(-2, 1);
\draw (-6, 0.5)--(-2, 0.5);

\node(Os) at (0.9,0) {$x_{\infty}$};
\node(Os) at (0.9,0.5) {$x_{t_{n'}}$};
\node(Os) at (0.9,1) {$x_{t_{n}}$};

\node(Os) at (-1.7, 0) {$x_{\infty}$};
\node(Os) at (-1.7,1) {$x_{t_{n'}}$};
\node(Os) at (-1.7,0.5) {$x_{t_{n}}$};

\node(OS) at (3, 1) {$\Phi(U)$};
\node(Os) at  (3, -1){$\Phi(U\cap B_{j_{t_{n}}})$};

\node(OS) at (-4, -1) {$\Phi(U)$};
\node(Os) at  (-4, 1){$\Phi(U\cap B_{j_{t_{n}}})$};
\end{tikzpicture}
\caption{}
\end{center}
\end{figure}

 From $(\ast\ast)$,  $\Sigma^{t_{n}}_{i_{n}}\cap U\subset L_{t_{n}}$ is connected and a graph over $B^{\Sigma_{\infty}}(p_{\infty})$, for $n$ large enough. Since $\partial B_{j_{t_{n}}}\cap U\subset L_{{t_{n}}}$ equals $\Sigma^{t_{n}}_{i_{n}}\cap U$, it is also  connected. Therefore $\Phi(\partial B_{j_{t_{n}}}\cap U)$ is the set $\mathbb{R}^{2}\times \{x_{t_{n}}\}\cap \Phi(U)$ for some $x_{t_{n}}\in \mathbb{R}$. In addition, $\Phi(\Sigma_{\infty}\cap U)$ equals $\mathbb{R}^{2}\times\{x_{\infty}\}\cap\Phi(U)$ for some $x_{\infty}\in \mathbb{R}$. Since $\lim\limits_{n\rightarrow \infty}p_{t_{n}}=p_{\infty}$, we have $\lim\limits_{n\rightarrow\infty}x_{t_{n}}=x_{\infty}$.  \par

The set $U\setminus \partial B_{j_{t_{n}}}$ has two components. Therefore, $\Phi(B_{j_{t_{n}}}\cap U)$ is $\Phi(U)\cap \{x| x_{3}> x_{t_{n}}\}$ or $\Phi(U)\cap \{x| x_{3}< x_{t_{n}}\}$. For $n$ large enough, there exists some $n'\neq n$ such that $\mathbb{R}^{2}\times \{x_{t_{n'}}\}\cap \Phi(U) \subset \Phi(B_{j_{t_{n}}}\cap U)$. This implies that $B_{j_{t_{n}}}\cap B_{j_{t_{n'}}}$ is non-empty.\par
Since $S$ consists of maximal elements in $(\{B^{t}_{i}\}, \subset)$, the set $B_{j_{t_{n}}}\cap B_{j_{t_{n'}}}$ must be empty  which leads to a contradiction. This finishes the proof.
\end{proof}

\subsection{The finiteness of $S$ implies Lemma \ref{covering}}
We will explain how to deduce Lemma \ref{covering} from the finiteness of $S$.\par

 \begin{proof} Since $S$ is  finite , we may assume that $\partial B_{j}$ intersects $\partial R_{k}(\epsilon)$ transversally for each $B_{j}\in S$.
 Remark that each $B_{j}$ is equal to some $B^{t}_{i}$ and $\partial B_{j}\cap \partial R_{k}(\epsilon)$ equals $\Sigma^{t}_{i}\cap \partial R_{k}(\epsilon)$.  Since each $\Sigma^{t}_{i}$ is properly embedded, $\{c_{i}\}_{i\in I}:=\partial R_{k}(\epsilon)\cap(\cup_{B_{j}\in S}\partial B_{j})$ has finitely many components.  Each component is an embedded circle.

The Vanishing property of $\mathscr{L}$ and Remark \ref{general} show that each $c_{i}$ is contractible in $\partial R_{k}(\epsilon)$ and bounds a unique closed disc $D_{i}\subset \partial R_{k}(\epsilon)$ (since $k\geq k_{0}$). The set $(D_{i}, \subset)$ is a partially ordered set. Let $\{D_{i'}\}_{i'\in I'}$ be the set of maximal elements. The set $I'$ is finite .\par

Since the boundary of $\partial R_{k}(\epsilon)\cap \overline{B}_{j}$ is a subset of $\partial B_{j}\cap \partial R_{k}(\epsilon)\subset \amalg_{i\in I}c_{i}$, it is  contained in $\amalg_{i'\in I'}D_{i'}$ for each $B_{j}\in S$.
 
Next we show that for any $B_{j}\in S$, $\partial R_{k}(\epsilon)\cap \overline{B_{j}}$ is contained in $\amalg_{i'\in I'}D_{i'}$.

If not, $\partial R_{k}(\epsilon)\setminus \amalg_{i'\in I'} D_{i'}$ is contained in $\partial R_{k}(\epsilon)\cap \overline{B}_{j}$ for some $B_{j}\in S$. This implies that the composition of two maps $\pi_{1}(\partial R_{k}(\epsilon)\setminus \amalg_{i'\in I'}D_{i'})\rightarrow \pi_{1}(\overline{B}_{j})\rightarrow\pi_{1}({R}_{k})$ is not a zero map. 
It is in contradiction with the fact that  the induced map $\pi_{1}(\overline{B}_{j})\rightarrow \pi_{1}({R}_{k})$ is trivial. We conclude that for each $B_{j}\in S$, $\partial R_{k}(\epsilon)\cap \overline{B}_{j}$ is contained in $\amalg_{i'\in I'}D_{i'}$.\par

Therefore, $\cup_{B_{j}\in S}\overline{B}_{j}\cap \partial R_{k}(\epsilon)$ is contained in $\amalg_{i'\in I'} D_{i'}$. From Proposition \ref{cover}, $\mathscr{L}\cap \partial R_{k}(\epsilon)$ is contained in a disjoint union of finite discs $\{D_{i'}\}_{i'\in I'}$. This completes the proof.
 \end{proof}

 \section*{Appendix A}
 
\noindent\textbf{Lemma} \ref{existence} \emph{Let $(\Omega, \partial \Omega) \subset (N, \partial N)$ be a 2-sided embedded disc with some closed sub-discs removed, where $N$ is a closed handlebody of genus $g>0$. Assume that each circle $\gamma_{i}$ is contractible in $\partial N$, where $\partial \Omega=\amalg_{i}\gamma_{i}$. Then $N\setminus \Omega$ has two connected components. Moreover, there is a unique component $B$ satisfying that the induced map $\pi_{1}(B)\rightarrow\pi_{1}(N)$ is trivial.  }

\begin{proof}  
As in the proof of Lemma 7.1 in \cite{W1}, we can conclude that $\Omega$ cuts $N$ into two components, $B_{1}$ and $B_{2}$.

Remark that each embedded circle $\gamma_{i}$ is contractible in $\partial N$ and bounds a unique closed disc $D_{i}\subset \p N$. Let us consider the surface $\hat{\Omega}:=\Omega\bigcup(\cup_{\gamma_{i}}D_{i})$. It is an immersed $2$-sphere in $N$. This induces that  the map $\pi_{1}(\Omega)\rightarrow\pi_{1}(\hat{\Omega})$ is trivial map. Therefore, the map $\pi_{1}(\Omega)\rightarrow \pi_{1}(N)$ is trivial. 

\vspace{2mm}

In the following, we show  the existence of $B$. 

Consider the partially ordered relationship over $\{D_{i}\}$ induced by inclusion. Therefore, $\cup_{i}D_{i}$ is equals to a disjoint union of maximal elements in $(\{D_{i}\}, \subset)$. The set $\partial N\setminus \cup_{i}D_{i}$ is a compact surface with some punctures. 

Therefore, the induced map $\pi_{1}(\partial N\setminus \cup_{i}D_{i})\rightarrow \pi_{1}(\partial N)$ is surjective. In addition, the induced map $\pi_{1}(\partial N)\rightarrow \pi_{1}(N)$ is also surjective. We can conclude that the composition of these two maps $\pi_{1}(\partial N\setminus \cup_{i}D_{i})\rightarrow \pi_{1}(N)$ is also surjective.

The set $\partial N\setminus \cup_{i}D_{i}$ is contained in one of the two components, $B_{1}$ and $B_{2}$, of $N\setminus \Omega$. Without loss of generality, we may assume that  $B_{1}$ contains $\partial N\setminus \cup_{i}D_{i}$. From above, the induced map $\pi_{1}(B_{1})\rightarrow \pi_{1}(N)$ is surjective. 

Let $G_{i}$ be the image of the map $\pi_{1}(B_{i})\rightarrow \pi_{1}(N)$.  Van-Kampen's Theorem gives an isomorphism between $\pi_{1}(N)$ and $\pi_{1}(B_{1})\ast_{\pi_{1}(\Omega)}\pi_{1}(B_{2})$.  Since the image of $\pi_{1}(\Omega)\rightarrow\pi_{1}(N)$ is trivial, $\pi_{1}(N)$ is isomorphic to $G_{1}\ast G_{2}$. Grushko's Theorem \cite{Gru} shows that 
$\text{rank}(G_{1})+\text{rank}(G_{2})=\text{rank}(\pi_{1}(N))$. (The rank of a group is the smallest cardinality of a generating set for the group.) From the last paragraph, the image, $G_{1}$, of the map $\pi_{1}(B_{1})\rightarrow \pi_{1}(N)$ is isomorphic to $\pi_{1}(N)$. That is to say, $\text{rank}(G_{1})=\text{rank}(\pi_{1}(N))$. Therefore,  $\text{rank}(G_{2})$ is equal to zero. Namely, $G_{2}$ is a trivial group. We know that $B:=B_{2}$ is the required candidate in the assertion.

The uniqueness is the same as the genus one case (See the proof of Lemma 7.1 in \cite{W1}).\end{proof}

\section*{Appendix B}

\noindent\textbf{Theorem}~\ref{proper}\emph{
Let $(M, g)$ be a complete oriented 3-manifold with positive scalar curvature $\kappa(x)$. Assume that $\Sigma$ is a complete (non-compact) stable minimal surface in $M$. Then, one has,$$\int_{\Sigma}\kappa(x)dv\leq 2\pi,$$where $dv$ is the volume form of the induced metric $ds^{2}$ over $\Sigma$. Moreover, if $\Sigma$ is an embedded surface, then $\Sigma$ is proper.}

\begin{proof} By [Theorem 2, Page 211] in \cite{SY}, $\Sigma$ is conformally diffeomorphic to $\mathbb{R}^{2}$.\par

Let us consider the Jacobi operator $L:=\Delta_{\Sigma}- K_{\Sigma}+(\kappa(x)+\frac{1}{2}|A|^{2})$, where $K_{\Sigma}$ is the Gaussian curvature of the metric $ds^{2}$ and $\Delta_{\Sigma}$ is the Laplace-Beltrami operator of $(\Sigma, ds^{2})$. From [Theorem 1, Page 201] in \cite{FS}, there exists a positive function $u$ on $\Sigma$ satisfying $L(u)=0$, since $\Sigma$ is a stable minimal surface. \par

Consider the metric $d\tilde{s}^{2}:=u^{2}ds^{2}$. Let $\tilde{K}_{\Sigma}$ be its sectional curvature and $d\tilde{v}$ its volume form. We know that 
$$\tilde{K}_{\Sigma}=u^{-2}(K_{\Sigma}-\Delta_{\Sigma}~log~ u)~~~~\text{and}~~~~~~~~d\tilde{v}=u^{2}dv.$$
By Fischer-Colbrie's work [Theorem 1, Page 126] in \cite{F}, $(\Sigma, d\tilde{s}^{2})$ is a complete surface with non-negative sectional curvature $\tilde{K}_{\Sigma} \geq 0$. By the Cohn-Vossen inequality \cite{CV}, one has

 $$\int_{\Sigma}\tilde{K}_{\Sigma}d\tilde{v}\leq 2\pi\chi(\Sigma),$$where $\chi(\Sigma)$ is the Euler characteristic of $\Sigma$.\par
 Since $L(u)=0$, one has that $\int_{B^{\Sigma}(0, R)} L(u)u^{-1}dv=0$, where $B^{\Sigma}(0,R)$ is the geodesic ball in $(\Sigma, ds^{2})$ centered at $0\in \Sigma$ with radius $R$. We deduce that
\begin{equation*}
	\begin{split}
	\int_{B^{\Sigma}(0, R)}\kappa(x)+\frac{1}{2}|A|^{2}dv &=\int_{B^{\Sigma}(0, R)}(K_{\Sigma}-u^{-1}\Delta_{\Sigma}~u)dv\\
	&=\int_{B^{\Sigma}(0, R)} K_{\Sigma}-(\Delta_{\Sigma}~log~u+u^{-2}|\nabla u|)dv\\
	&\leq \int_{B^{\Sigma}(0, R)}u^{-2}(K_{\Sigma}-\Delta_{\Sigma} log~u)u^{2}dv\\
	&=\int_{B^{\Sigma}(0, R))}\tilde{K}_{\Sigma}d\tilde{v}\\
	&\leq \int_{\Sigma} \tilde{K}_{\Sigma}d\tilde{v}
	\end{split}
	\end{equation*}
We know that $\chi(\Sigma)=1$, since $\Sigma$ is diffeomorphic to $\mathbb{R}^{2}$. Combining these two inequalities above and taking $R\rightarrow \infty$, we have that, 
$$\int_{\Sigma}\kappa(x)+\frac{1}{2}|A|^{2} dv\leq 2\pi.$$	

Suppose that $\Sigma$ is not proper.  There is an accumulation point $p$ of $\Sigma$ so that the set $B(p, r/2)\cap \Sigma$ is a non-compact closed set in $\Sigma$. Namely, it is unbounded in $(\Sigma, ds^2)$.
Hence, there is a sequence$\{p_{k}\}$ of points in $\Sigma\cap B(p, r/2)$ going to infinity in $(\Sigma, ds^2)$. 

Therefore, we may assume that the geodesic discs $\{B^{\Sigma}(p_k, r/2)\}_k$ in $\Sigma$ are disjoint.

Define two constants $r_{0}:=\frac{1}{2}\min\{r, i_{0}\}$ and $K:=\sup_{x\in B(p, r)}|K_{M}(x)|$ where $i_{0}:=\inf_{x\in B(p, r)} \text{Inj}_{M}(x)$ and $K_{M}$ is the sectional curvature. The geodesic disc $B^{\Sigma}(p_{k}, r_{0}/2)$ is in $B(p, r)$. 

Applying [Theorem 3, Appendix, Page 139] of \cite{FKR}  to the geodesic disc $ B^{\Sigma}(p_k, r_0/2)\subset B(p, r)$, we have 
  $$\text{Area}( B^{\Sigma}(p_{k}, r_{0}/2))\geq C(i_{0}, r_{0}, K).$$ 
 This leads to a contradiction as follows:
\begin{equation*}
\begin{split}
2\pi &\geq \int_{\Sigma}\kappa(x)dv \geq \int_{B(p, r)\cap \Sigma} \kappa(x)dv\\
      &\geq \sum_{k} \int_{B^{\Sigma}(p_{k},r_{0}/2)} \kappa(x)\\
      &\geq \inf_{x\in B(p, r)} \kappa(x)\cdot\sum_{k} \text{Area}(B^{\Sigma}(p_{k}, r_{0}/2))\\ 
      &\geq \inf_{x\in B(p, r)} \kappa(x)\cdot\sum_{k} C= \infty
      \end{split}
\end{equation*}
\end{proof}

\section*{Appendix C: Example} 

There are infinitely many  contractible $3$-manifolds with non-trivial fundamental group at infinity. In this Appendix, we construct such a $3$-manifold $M$ and analyse its topology. We will prove that this $3$-manifold has no complete metric of positive scalar curvature. 

\vspace{2mm}
\noindent{\textbf{C.1~ The construction of $M$}.} Before constructing the $3$-manifold, let us introduce a notation. A handlebody $N\subset \mathbb{S}^{3}$ of genus $g$ is said to be unknotted in $\mathbb{S}^{3}$ if its complement in $\mathbb{S}^{3}$ is a handlebody of genus $g$. 

\begin{figure}[H]
\centering{
\def\svgwidth{\columnwidth}
{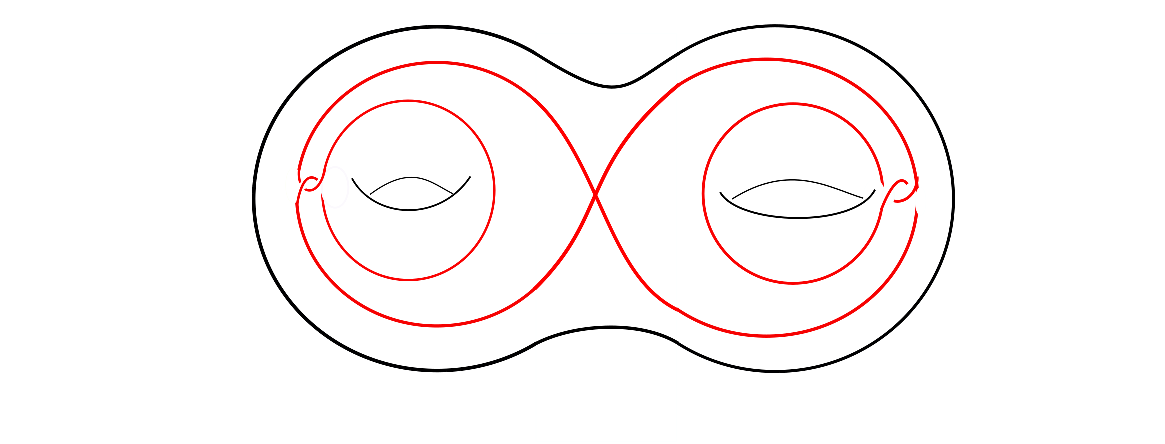}
\caption{}
\label{Fig1}
}
\end{figure}
 Choose an unknotted  handlebody $W_{0}\subset\mathbb{S}^{3}$ of genus two.  Take a second handlebody $W_{1}\subset \text{Int}~W_{0}$ of genus two which is a tubular neighborhood of the  curve in Figure \ref{Fig1}.  Then, embed another handlebody $W_{2}$ of genus two inside $W_{1}$ in the same way as $W_{1}$ lies in $W_{0}$ and so on infinitely many times.  Therefore, we obtain a decreasing family $\{W_{k}\}$ of handlebodies of genus two.

The manifold $M$ is defined as $M:=\mathbb{S}^{3}\setminus \cap_{k=0}^\infty W_{k}$. It is an open manifold. 

\begin{figure}[H]
\centering{
\def\svgwidth{\columnwidth}
{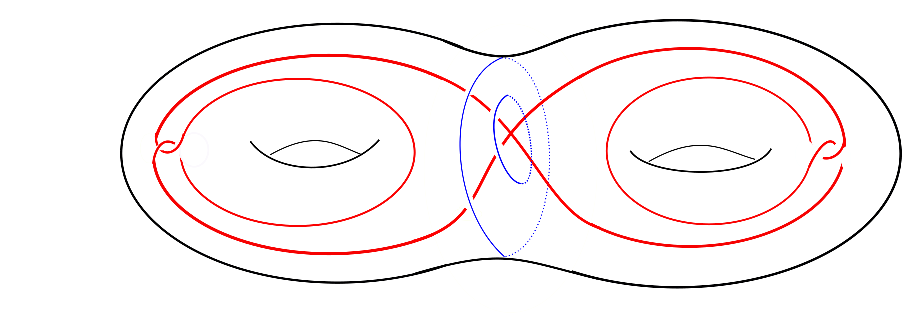}
\caption{}
\label{Fig2}
}
\end{figure}

We have that each $W_{k}$ is unknotted in $\mathbb{S}^{3}$. Namely,  the complement $N_{k}$ of $W_{k}$ in $\mathbb{S}^{3}$ is a handlebody of genus two. Therefore, $M$ can be written as an increasing union of handlebodies $\{N_{k}\}_{k}$ of genus two. Furthermore,  each  $N_{k}$ lies in $N_{k+1}$ as in Figure \ref{Fig2}. (The set $K_{k}$ is the core of $N_{k}$.)

 Each $N_{k}$ is homotopically trivial in $N_{k+1}$. We can conclude that $M$ is a contractible $3$-manifold. 

\vspace{2mm}

\noindent \textbf{C.2 The topological property of $M$.} In this part, we first show that the fundamental group at infinity of $M$ is non-trivial. As a consequence, $M$ is not homeomorphic to $\mathbb{R}^3$.  In the manifold $M$, there is a properly embedded plane. This plane cuts $M$ into two Whitehead manifolds.  

\vspace{2mm}

First, we see from Figure \ref{Fig1} that $W_{k}$ is an effective handlebody relative to $W_{k+1}$ for each $k$. From Lemma \ref{injective}, the map $\pi_{1}(\partial W_{k})\rightarrow \pi_{1}(\overline{W_{k}\setminus W_{k+1}})$ is injective. In addition, the set $\overline{W_{k}\setminus W_{k+1}}$ is equal to $\overline{N_{k+1}\setminus N_{k}}$. Therefore, we conclude that for each $k$, the map $\pi_{1}(\partial N_{k})\rightarrow \pi_{1}(\overline{N_{k+1}\setminus N_{k}})$ is injective. 

Second, from Figure \ref{Fig2}, we show that each $N_{k+1}$ is an effective handlebody relative to $N_{k}$. By Lemma \ref{injective}, the map $\pi_{1}(\partial N_{k+1})\rightarrow \pi_{1}(\overline{N_{k+1}\setminus N_{k}})$ is injective. 

\vspace{2mm}

As in the genus one case, for each $k$, the maps $\pi_{1}(\partial N_{k})\rightarrow \pi_{1}(\overline{M\setminus N_{k}})$ and  $\pi_{1}(\partial N_{k})\rightarrow \pi_{1}(\overline{N_{k}\setminus N_{0}})$ are both injective. That is to say, the family $\{N_{k}\}$ has Property (H). 

\vspace{2mm}
 
Pick the separating meridian $\gamma_{k}\subset \partial N_{k}$ as in Figure \ref{Fig2} . From Figure \ref{Fig2}, for each $k$, $\gamma_{k}$ is homotopic to $\gamma_{k+1}$ in $\overline{N_{k+1}\setminus N_{k}}$.  Since $\{N_{k}\}$ satisfies Property (H), the map $\pi_{1}(\partial N_{k})\rightarrow \pi_{1}(\overline{M\setminus N_{0}})$ is injective (See Remark \ref{inj}). That is to say, for $k>0$, $\gamma_{k}$ is non-contractible in $M\setminus N_{0}$. 

From Remark \ref{element}, the sequence of $\{\gamma_{k}\}$ gives a non-trivial element in $\pi^\infty_{1}(M)$. Since $\pi^\infty_{1}(M)$ is non-trivial, $M$ is not simply-connected at infinity. In particular, $M$ is not homeomorphic to $\mathbb{R}^{3}$. 

\vspace{2mm}

Next, we construct the properly embedded plane in $M$ from the sequence $\{\gamma_{k}\}_{k}$. 

 Choose an embedded annulus $A_{k}\subset \overline{N_{k+1}\setminus N_{k}}$ with boundary $\gamma_{k}\amalg \gamma_{k+1}$. Let $D_{0}\subset N_{0}$ be a meridian disc with boundary $\gamma_{0}$. We define the plane $P$ as 
 \begin{equation*}
 P:=\cup_{k\geq 0} A_{k}\cup D_{0}. 
 \end{equation*}
  The plane $P$ cuts $M$ into two contractible $3$-manifolds $M'$ and $M''$. In addition, the intersection $P\cap N_{k}$ is a separating meridian disc of $N_{k}$ with boundary $\gamma_{k}$.

    \begin{figure}[H]
\centering{
\def\svgwidth{\columnwidth}
{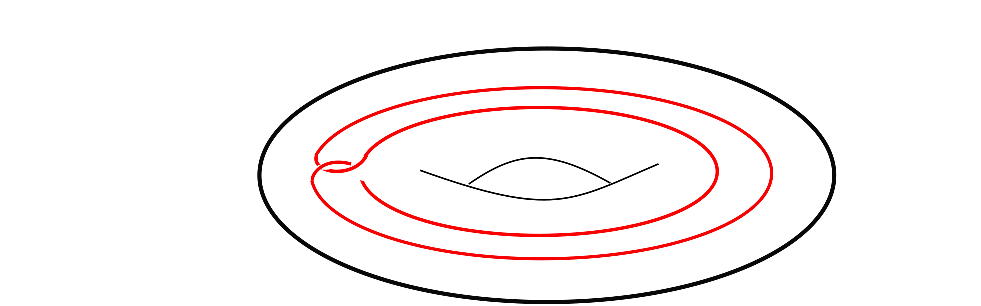}
\caption{}
\label{Fig3}
}
\end{figure}
  
  From the sequence $\{N_{k}\}$, we obtain two increasing families, $\{N'_{k}\}$ and $\{N''_{k}\}$, of solid tori in $M$ satisfying that 
  \begin{itemize}
  \item $M'=\cup_{k} N'_{k}$ and $M''=\cup_{k} N''_{k}$;
  \item the set $N_{k}\setminus (N'_{k}\amalg N''_{k})$ is a tubular neighborhood of the meridian disc $P\cap N_{k}$. 
  \end{itemize}

  Furthermore, each $N'_{k}$ is embedded into $N'_{k+1}$ as in Figure \ref{Fig3}. We see that $M'$ is homeomorphic to the Whitehead manifold. Similarly, the contractible $3$-manifold $M''$ is also homeomorphic to the Whitehead manifold. Therefore, $P$ cuts $M$ into two Whitehead manifolds.

\noindent\textbf{C.3  Non-existence of PSC metrics.}
In this part, we show that the manifold $M$ has no complete metric of positive scalar curvature. 

We argue by contradiction. Suppose that $M$ has a complete metric of positive scalar curvature. As in Section 5.1, there is a family of laminations $\{\mathscr{L}_{k}\}_k$ sub-converging toward a stable minimal lamination $\mathscr{L}:=\cup_{t\in \Lambda} L_{t}$. 

Since $\pi_{1}^\infty(M)$ is non-trivial,  some leaf in $\mathscr{L}$ may not satisfy the Vanishing property for $\{N_{k}\}_{k}$. To overcome it, we attempt to find a new family of handlebodies with Property (H). 

\vspace{2mm}

We know that $M'=\cup_{k}N'_{k}$ is homeomorphic to the Whitehead manifold. The geometric index $I(N'_{k}, N'_{k+1})$ is equal to $2$. (See Section 2 of \cite{W1}). From Lemma 2.10 of \cite{W1}, the maps $\pi_{1}(\partial N'_{k})\rightarrow \pi_{1}(\overline{M\setminus N'_{0}})$ and $\pi_{1}(\partial N'_{k})\rightarrow \pi_{1}(\overline{N'_{k}\setminus N'_{0}})$ are both injective. Therefore, the family $\{N'_{k}\}$ satisfies Property (H). 

\vspace{2mm}

In addition, each leaf $L_{t}$ in $\mathscr{L}$ satisfies Property P (See Definition 3.3 of \cite{W1}). That is to say, for any circle $\gamma\subset L_{t}\cap \partial N'_{k}$, one of the following holds: 
\begin{itemize}
\item $\gamma$ is homotopically trivial in $\partial N'_{k}$;
\item for $l\leq k$, $D\cap \text{Int}~ N'_{l}$ has at least $I(N'_{l}, N'_{k})$ components intersecting $N_{0}$
\end{itemize}where $D\subset L_{t}$ is the unique disc with boundary $\gamma$ and the geometric index $I(N'_{l}, N'_{k})$ is equal to $2^{k-l}$. 

\vspace{2mm}

 In the following, we consider the geometry of the leaves intersecting $M'$. As Lemma 6.1 of \cite{W1}, we know the following claim. 

\vspace{2mm}

\textbf{Claim}: \emph{$\mathscr{L}$ satisfies the Vanishing property with respect to $\{N'_{k}\}_{k}$.} 

\vspace{2mm}

The proof of this claim is the same as Lemma 6.1 of \cite{W1}. Let us explain the proof of the claim. 

We argue  by contradiction. We suppose that there exists a sequence of increasing integers $\{k_{n}\}_{n}$ such that :\par
\vspace{2mm}
{\emph{ for each $k_{n}$, there exists a minimal surface $L_{t_{n}}$ in $\{L_{t}\}_{t\in \Lambda}$ and an embedded curve $c_{k_{n}}\subset L_{t_{n}}\cap\partial N'_{k_{n}}$ which is not contractible in $\partial N'_{k_{n}}$.} }\par
\vspace{2mm}
\noindent Since $\lim\limits_{n\rightarrow\infty}k_{n}=\infty$,  $\lim\limits_{n\rightarrow\infty} I(N'_{1}, N'_{k_{n}})=\infty$.

Since $(M, g)$ has positive scalar curvature, $L_{t_{n}}$ is homeomorphic to $\mathbb{R}^{2}$. Then,  there exists a unique disc $D_{n}\subset L_{t_{n}}$ with boundary $c_{k_{n}}$. From the above property,  we see that $D_{n}\cap N'_{1}$ has at least $I(N'_{1}, N'_{k_{n}})$ components intersecting $N'_{0}$, denoted by $\{\Sigma_{j}\}_{j=1}^{m}$.

Define the constants $r:=d^{M}(\partial N'_{0},\partial N'_{1})$, $C:=\inf_{x\in N'_{1}}\kappa(x)$, $K:=\sup_{x\in N'_{1}}|K_{M}|$ and $i_{0}:=\inf_{x\in N'_{1}}(\text{Inj}_{M}(x))$ , where $K_{M}$ is the sectional curvature of $(M, g)$ and  $\text{Inj}_{M}(x)$ is the injective radius at $x$ of $(M, g)$.\par
 Choose $r_{0}=\frac{1}{2}\min\{i_{0}, r\}$ and $x_{j}\in \Sigma_{j}\cap N'_{0}$, then $B(x_{j}, r_{0})$ is in $ N'_{1}$. We apply [Lemma 1, Page 445] in \cite{YM} to the minimal surface $(\Sigma_{j}, \partial\Sigma_{j})\subset (N'_{1}, \partial N'_{1})$. Hence, one has that 
$$\text{Area}(\Sigma_{j}\cap B(x_{j}, r_{0}))\geq C_{1}(K, i_{0}, r_{0}).$$
From Theorem \ref{proper}, we have:
\begin{equation*}
\begin{split}
2\pi &\geq \int_{L_{t_{n}}} \kappa(x) dv\geq \sum_{j=1}^{m}\int_{\Sigma_{j}} \kappa(x) dv \geq \sum_{j=1}^{m}\int_{\Sigma_{j}\cap B(x_{j}, r_{0})} \kappa(x) dv\\
       &\geq \sum_{j=1}^{m} C\text{Area}(\Sigma_{j}\cap B(x_{j}, r_{0})) \\
       &\geq CC_{1}m\geq CC_{1}I(N'_{1}, N'_{k_{n}})
       \end{split}
\end{equation*}
This contradicts  the fact that $\lim\limits_{n\rightarrow\infty}I(N'_{1}, N'_{k_{n}})=\infty$ and completes the proof of the claim.

\vspace{2mm}

In addition, since none of the $N'_{k}$ is contained in a $3$-ball,  we use Corollary \ref{inter} to know that if $N_{j}$ contains $N'_{k}$, the intersection $\mathscr{L}_{j}\cap \partial N'_{k}$ has at least one meridian of $R_{k}$. 

To sum up, the family $\{N'_{k}\}_{k}$ satisfies a) and b) in Section 7. That is to say, 

a) $\mathscr{L}$ satisfies the Vanishing property for $\{N'_{k}\}_{k}$;

b) if the handlebody $N_{j}$ contains $N'_{k}$, the intersection $\mathscr{L}_{j}\cap \p N'_{k}$ has at least one meridian of $N'_{k}$. 

The remaining proof is the same as the proof of Theorem \ref{A} in Sections 7 and 8.

\section*{Acknowledgement}
I would like to express my deep gratitude to my supervisor, G\'erard Besson for suggesting this problem, also for his endless support, constant encouragement. I also thank Harold Rosenberg for many enlightening discussions about minimal surfaces theory. I am  thankful to Laurent Mazet for the interest and for many valuable conversations. I am also very grateful to the referees for many helpful comments that improved the exposition of this paper. This research is supported by ERC Advanced Grant 320939, GETOM and partially by  the Special Priority Program SPP 2026 `` Geometry at infinity ".

\bibliographystyle{alpha}
\bibliography{positive2}

\newcommand{\etalchar}[1]{$^{#1}$}
\begin{thebibliography}{BBB{\etalchar{+}}10}

\bibitem[And85]{And}
Michael Anderson.
\newblock Curvature estimates for minimal surfaces in $3 $-manifolds.
\newblock In {\em Annales scientifiques de l'{\'E}cole Normale Sup{\'e}rieure},
  volume~18, pages 89--105. Elsevier, 1985.

\bibitem[BBB{\etalchar{+}}10]{BBBMP}
Laurent Bessi{\`e}res, G{\'e}rard Besson, Michel Boileau, Sylvain Maillot, and
  Joan Porti.
\newblock {\em Geometrisation of 3-manifolds}, volume~13.
\newblock European Mathematical Society, 2010.

\bibitem[CM04]{CM}
Tobias {Colding} and William {Minicozzi}.
\newblock The space of embedded minimal surfaces of fixed genus in a
  3-manifold. iv: Locally simply connected.
\newblock {\em Ann. Math. (2)}, 160(2):573--615, 2004.

\bibitem[CM11]{CM1}
Tobias Colding and William Minicozzi.
\newblock {\em A course in minimal surfaces}, volume 121.
\newblock American Mathematical Soc, 2011.

\bibitem[{Coh}35]{CV}
Stefan {Cohn-Vossen}.
\newblock {K\"urzeste Wege und Totalkr\"ummung auf Fl\"achen.}
\newblock {\em {Compos. Math.}}, 2:69--133, 1935.

\bibitem[CWY10]{CWY}
Stanley Chang, Shmuel Weinberger, and Guoliang Yu.
\newblock Taming 3-manifolds using scalar curvature.
\newblock {\em Geometriae Dedicata}, 148(1):3--14, 2010.

\bibitem[CZ06]{cao-zhu}
Huai-Dong Cao and Xi-Ping Zhu.
\newblock Hamilton-perelman's proof of the poincar\'e conjecture and the
  geometrization conjecture.
\newblock {\em arXiv preprint math/0612069}, 2006.

\bibitem[FCS80]{FS}
D.~Fischer-Colbrie and Richard Schoen.
\newblock The structure of complete stable minimal surfaces in 3-manifolds of
  non-negative scalar curvature.
\newblock {\em Communications on Pure and Applied Mathematics}, 33(2):199--211,
  1980.

\bibitem[{Fis}85]{F}
D.~{Fischer-Colbrie}.
\newblock {On complete minimal surfaces with finite Morse index in three
  manifolds.}
\newblock {\em {Invent. Math.}}, 82:121--132, 1985.

\bibitem[Fre96]{FKR}
Katia~Rosenvald Frensel.
\newblock Stable complete surfaces with constant mean curvature.
\newblock {\em Boletim da Sociedade Brasileira de Matem{\'a}tica},
  27(2):129--144, 1996.

\bibitem[GL83]{GL}
Mikhael Gromov and Blaine Lawson.
\newblock Positive scalar curvature and the dirac operator on complete
  riemannian manifolds.
\newblock {\em Publications Math{\'e}matiques de l'Institut des Hautes
  {\'E}tudes Scientifiques}, 58(1):83--196, 1983.

\bibitem[Gru40]{Gru}
Igor~Aleksandrovich Grushko.
\newblock On the bases of a free product of groups.
\newblock {\em Mat. Sbornik}, 8:169--182, 1940.

\bibitem[Hat00]{HA}
Allen Hatcher.
\newblock Notes on basic 3-manifold topology, 2000.

\bibitem[McM61]{Mc1}
DR~McMillan.
\newblock Cartesian products of contractible open manifolds.
\newblock {\em Bulletin of the American Mathematical Society}, 67(5):510--514,
  1961.

\bibitem[McM62]{Mc}
DR~McMillan.
\newblock Some contractible open 3-manifolds.
\newblock {\em Transactions of the American Mathematical Society},
  102(2):373--382, 1962.

\bibitem[MRR02]{MRR}
William Meeks, Antonio Ros, and Harold Rosenberg.
\newblock {\em The global theory of minimal surfaces in flat spaces}.
\newblock Springer, 2002.

\bibitem[MT07]{MT}
John Morgan and Gang Tian.
\newblock {\em Ricci flow and the Poincar{\'e} conjecture}, volume~3.
\newblock American Mathematical Soc., 2007.

\bibitem[MY80]{YM}
William Meeks and Shing-Tung Yau.
\newblock Topology of three dimensional manifolds and the embedding problems in
  minimal surface theory.
\newblock {\em Annals of Mathematics}, 112(3):441--484, 1980.

\bibitem[MY82]{YM1}
William Meeks and Shing-Tung Yau.
\newblock The existence of embedded minimal surfaces and the problem of
  uniqueness.
\newblock {\em Mathematische Zeitschrift}, 179(2):151--168, 1982.

\bibitem[Mye99]{Myers}
Robert Myers.
\newblock Contractible open 3-manifolds which non-trivially cover only
  non-compact 3-manifolds.
\newblock {\em Topology}, 38(1):85--94, 1999.

\bibitem[Per02a]{P1}
Grisha Perelman.
\newblock The entropy formula for ricci flow and its geometric applications.
\newblock {\em arXiv preprint math/0211159}, 2002.

\bibitem[Per02b]{P2}
Grisha Perelman.
\newblock Ricci flow with surgery on three-manifolds.
\newblock {\em arXiv preprint math/0303109}, 2002.

\bibitem[Per03]{P3}
Grisha Perelman.
\newblock Finite extinction time for the solutions to the ricci flow on certain
  three-manifolds.
\newblock {\em arXiv preprint math/0307245}, 2003.

\bibitem[Rol03]{Rol}
Dale Rolfsen.
\newblock {\em Knots and links}, volume 346.
\newblock American Mathematical Soc., 2003.

\bibitem[Rot12]{Rot}
Joseph Rotman.
\newblock {\em An introduction to the theory of groups}, volume 148.
\newblock Springer Science \& Business Media, 2012.

\bibitem[Sch83]{Schoen}
Richard Schoen.
\newblock Estimates for stable minimal surfaces in three dimensional manifolds.
\newblock In {\em Seminar on minimal submanifolds}, volume 103, pages 111--126.
  Princeton University Press Princeton, NJ, 1983.

\bibitem[Sta72]{S}
John Stallings.
\newblock {\em Group theory and three-dimensional manifolds}.
\newblock Yale University Press, 1972.

\bibitem[SY79]{SY1}
Richard Schoen and Shing-Tung Yau.
\newblock On the structure of manifolds with positive scalar curvature.
\newblock {\em Manuscripta mathematica}, 28(1-3):159--183, 1979.

\bibitem[SY82]{SY}
Richard Schoen and Shing~Tung Yau.
\newblock Complete three-dimensional manifolds with positive ricci curvature
  and scalar curvature.
\newblock In {\em Seminar on Differential Geometry}, volume 102, pages
  209--228. Princeton Univ. Press Princeton, NJ, 1982.

\bibitem[Wan19]{W1}
Jian Wang.
\newblock Contractible 3-manifolds and positive scalar curvatures ({I}).
\newblock {\em arXiv:1901.04605, accepted By Journal of Differential Geometry},
  2019.

\bibitem[Whi35]{Wh}
J.H.C. Whitehead.
\newblock A certain open manifold whose group is unity.
\newblock {\em The Quarterly Journal of Mathematics}, 6(1):268--279, 1935.

\end{thebibliography}

\end{document}